\documentclass[11pt]{article}

\usepackage{graphics,color}

\usepackage{amssymb,amsmath, amsthm}
\usepackage{graphicx}

\newtheorem{thm}{Theorem}[section]
\newtheorem{theorem}[thm]{Theorem}

\newtheorem{lem}[thm]{Lemma}
\newtheorem{remark}[thm]{Remark}
\newtheorem{lemma}[thm]{Lemma}
\newtheorem{cor}[thm]{Corrollary}

\newtheorem{definition}[thm]{Definition}

\newenvironment{pfoflem}[1]
{\par\vskip2\parsep\noindent{\sc Proof of Lemma\ #1. }}{{\hfill
$\Box$}
\par\vskip2\parsep}

\newcommand{\dominotilings}{X}
\newcommand{\prob}{{\mu}}
\newcommand{\tiles}{\text{Tiles}}
\newcommand{\cycles}{\text{Cycles}}
\newcommand{\set}{B\cap \tilde x_{2^kn}}

\newcommand{\markminusone}{-1}
\newcommand{\markzero}{0}
\newcommand{\markone}{1}
\newcommand{\marktwo}{2}
\newcommand{\markthree}{3}
\newcommand{\markfour}{4}
\newcommand{\markfive}{5}
\newcommand{\marksix}{6}
\newcommand{\markseven}{7}
\newcommand{\markeight}{8}

\newcommand{\newmarkminusone}{\circle{.1}}
\newcommand{\newmarkzero}{{\circle*{.1}}}
\newcommand{\newmarkone}{\circle*{.1}}
\newcommand{\newmarktwo}{\circle{.1}}
\newcommand{\newmarkthree}{\circle{.1}}
\newcommand{\newmarkfour}{\circle{.1}}
\newcommand{\newmarkfive}{\circle*{.1}}
\newcommand{\newmarksix}{\circle{.1}}
\newcommand{\newmarkseven}{\circle*{.1}}
\newcommand{\newmarkeight}{\circle*{.1}}

\newcommand{\colorprob}{{\mathbb P}}

\newcommand{\cyl}{Cyl}

\newcommand{\num} {\# \text{ of domino tilings}}
\newcommand{\agree} {\text{Agree}}
\newcommand{\agreedcolors} {\text{Colors that Agree}}
\newcommand{\colors} {\text{Total Colors}}
\newcommand{\sq} {\text{Square}}
\newcommand{\ann} {\text{Annulus}}

\newcommand{\R}{\mathbb{R}}
\newcommand{\Z}{\mathbb{Z}}
\newcommand{\N}{\mathbb{N}}

\newcommand{\halfz}{\text{Colored Points}_x}
\newcommand{\halfA}{\text{Colored Points}_A}
\newcommand{\halfD}{\text{Colored Points}_D}
\newcommand{\halfDprime}{\text{Colored Points}_{D'}}
\newcommand{\halfB}{\text{Colored Points}_B}

\newcommand{\colorsD}{\text{Colors(D)}}
\newcommand{\be}{\begin{equation}}
\newcommand{\ee}{\end{equation}}

\title{Subshifts of finite type which have completely positive entropy}
\date{\today}
\author{Christopher Hoffman}
\begin{document}
\maketitle
\thanks{\begin{center}In memory of Dan Rudolph, 1949-2010\end{center}}

\abstract{ Domino tilings have been studied extensively for both
their statistical properties \cite{cohn}, \cite{kenyon2},
\cite{fields} and their dynamical properties \cite{bp}.  We
construct a subshift of finite type using matching rules for several
types of dominos.  We combine the previous results about domino
tilings to show that our subshift of finite type has a measure of maximal entropy with which the subshift has completely
positive entropy but is not isomorphic to a Bernoulli shift. }

\section{Introduction}

Subshifts of finite type are a fundamental object of study in
dynamics. A {\bf $\Z^d$ subshift of finite type} is defined by a
finite set $A$ and a finite list of forbidden words
$\text{Forbidden} \subset A^{[-n,n]^d}$.  The state space is $S
\subset A^{\Z^d}$ such that
$$S=\left\{s \in A^{\Z^d}:\  T_w(s) \not \in \text{ Forbidden }\  \forall w \in \Z^d \right\}$$
where the shift maps $T_w: A^{\Z^d} \to A^{\Z^d}$ are defined by
$(T_w(a))_v=a_{v+w}.$

 The {\bf topological entropy}, $k$, of a subshift of
finite type is defined to be
$$h=\lim_{k \to \infty} \frac{\log_2 \text{Admissible($k$)}}{(2k+1)^d}$$
where $\text{Admissible($k$)}$ be the number of words in
$A^{[-k,k]^d}$ that do not contain a forbidden word.

Subshifts of finite type are fundamentally topological objects.
However the study of subshifts of finite type often includes measure-theoretic question. This is possible because for every subshift of
finite type there exists an invariant measure with measure-theoretic entropy equal to
the topological entropy \cite{lindbook}. Using these measures
we can study the ergodic theoretic properties of a subshift of
finite type with respect to its measures of maximal entropy.

The ergodic theoretic properties of one dimensional subshifts of
finite type are well understood. The state space $S$ is non-empty if
and only if it contains periodic points.  Also there is an algorithm
to calculate the topological entropy of a one dimensional subshift
of finite type. Under very mild conditions a one dimensional
subshifts of finite type has a unique measure of maximal entropy.
Finally if a one dimensional subshift of finite type is mixing with
respect to its measure of maximal entropy then it is measurably
isomorphic to a Bernoulli shift.  See \cite{lindbook} for more
details about subshifts of finite type.

In contrast two (and higher) dimensional subshifts of finite type
may have very different behaviors.  In fact none of the properties
listed above necessarily apply to all two dimensional subshifts of
finite type. For instance given an alphabet and a list of forbidden words
it may be a difficult
problem to determine if the state space $S$ is empty or not. In fact there exists subshifts of finite type for
which $S$ is not empty but $S$ contains no periodic points. Because
of this there is no algorithm which can determine whether a subshift
of finite type has nonempty state space
\cite{wangtile}. It can also be difficult to
calculate the topological entropy of a subshift of finite type even
for some of the simplest subshifts of finite type (such as the hard
sphere model).

The measure theoretic properties of two dimensional subshifts of
finite type can also be quite complicated.
 Ledrappier showed that there are $\Z^2$ subshifts of finite type which
 are mixing but not mixing of all orders \cite{L}.  It remains a long
 standing open question as to whether there are actions of $\Z$
 which are mixing but not mixing of all orders.
 Burton and Steif used ideas from statistical
 physics to show that there are strongly irreducible subshifts of
 finite type with multiple measures of maximal entropy, and
 these measures of maximal entropy are not weak mixing \cite{bs}.

One particular subshift of finite type that has been very well
studied is the domino tiling of the plane \cite{bp} \cite{kenyon2}.
We will construct a subshift of finite type that is a variant of the
domino tiling of the plane which we call the {\bf colored domino
tiling}. We make use of some of the results about the domino tiling
of the plane to analyze the ergodic theoretic properties of the
colored domino tiling.  A transformation has {\bf completely
positive entropy} if every nontrivial factor of the transformation
has positive entropy.  We will show that the colored domino tiling
has completely positive entropy but is not isomorphic to a Bernoulli
shift.

The rest of this paper is organized as follows.
 In the next section we review some of the results about the domino tiling.
 In Section \ref{construction} we construct a zero entropy extension of the domino tiling which we call the colored domino tiling.
 In Section \ref{subshift} we construct the subshift of finite type.
 In Section \ref{entropy} we calculate the entropy of the subshift of finite type and identify a measure of maximal entropy.
 In Section \ref{same} we show the connection between the two processes.
 Then in Section \ref{isk} we show that our subshift of finite type has completely positive entropy.
 Finally in Section \ref{notbernoulli} we show that it is not isomorphic to a Bernoulli shift.

We conclude this section with an open question.  The subshift of finite type that we construct has multiple measures of maximal entropy.  We show that with respect to one of them the subshift has completely positive entropy but is not isomorphic to a Bernoulli shift.  With other measures of maximal entropy the subshift is isomorphic to a Bernoulli shift.  This leads us to the question:  Does there exist a subshift of finite type which has a unique measure of maximal entropy and with respect to that measure the subshift has completely positive entropy but is not isomorphic to a Bernoulli shift?  We believe the answer to be yes and that the techniques in this paper could be extended to construct such an example.

\section{Domino Tilings and the Height Function}
\label{domino}

A {\bf domino tiling} is a map $x$ from $(\Z+\frac12)^2 \to
(\Z+\frac12)^2$ such that
\begin{enumerate}
\item $||x(u)-u||_1=1$ for all $u$ and
\item $x(u)=v$ iff $x(v)=u$.
\end{enumerate}
We call this a domino tiling because we can think of this as for
each $u\in (\Z+\frac12)^2$ there is a 2 by 1 domino whose two
squares are centered at $u$ and $x(u)$.

We let $\dominotilings$ be the space of all domino tilings.  There are two natural shift operations on $\dominotilings$ given by
\begin{enumerate}
\item $T_{\text{left}}(x)(u)=x(u+(1,0))-(1,0)$ for all $x \in \dominotilings$ and $u\in (\Z+\frac12)^2$.
\item $T_{\text{down}}(x)(u)=x(u+(0,1))-(0,1)$ for all $x \in \dominotilings$ and $u\in (\Z+\frac12)^2$.
\end{enumerate}


Burton and Pemantle studied the ergodic theoretic properties of the
domino tiling \cite{bp}.

\begin{theorem} \cite{bp} \label{dominoisbernoulli}
There is a unique measure of maximal entropy $\prob$ on $\dominotilings$. The action   $(\dominotilings,\prob,T_{\text{left}},T_{\text{down}})$ is isomorphic to a $\Z^2$ Bernoulli shift.
\end{theorem}

The {\bf height function} $h_x$ of a domino tiling $x$ is an integer
valued function on $\Z^2$. It changes by 1 along each edge of the
graph that is on the boundary of a domino and changes by 3 along
each edge of the graph that bisects a domino.

More precisely for a tiling $x$ the height function $h_x:\Z^2 \to \Z$
such that for any $z, z' \in \Z^2$ with $||z-z'||_1=1$

\begin{enumerate}
\item $|h_x(z)-h_x(z')|=3$ if there is a  $u$ such that $u+x(u)=z+z'$
\item $|h_x(z)-h_x(z')|=1$ if there is no $u$ such that $u+x(u)=z+z'$.
\end{enumerate}

If we further require that $h_x(0,0)=0$ then there are exactly two
choices for $h$.  We pick one arbitrarily by first putting a
checkerboard pattern on the plane with a white in the square between
(0,0) and (1,1).  Then we say pick the height function so that if
$v,w \in \Z^2$ and the edge between $v$ and $w$ bisects a domino in
$x$ then $h_x(v)-h_x(w)=3$ if moving from $v$ to $w$ there is a
white square on your left and $h_x(v)-h_x(w)=-3$ if moving from $v$
to $w$ there is a white square on your right.
 The height function has been extensively studied \cite{height1} \cite{kenyon1}
\cite{kenyon2} \cite{kenyon4} \cite{sheffield} \cite{thurston}.
In Figure \ref{bayh} we show an example of a domino tiling $y$ and its corresponding height function.

\setlength{\unitlength}{12mm}
\newsavebox{\mysquare}
\savebox{\mysquare}{\textcolor{yellow}{\rule{12mm}{12mm}}}

\begin{figure}

\begin{picture}(12,10)
 \put(2,0){\usebox{\mysquare}} \put(4,0){\usebox{\mysquare}} \put(8,0){\usebox{\mysquare}}

\put(0,2){\usebox{\mysquare}} \put(2,2){\usebox{\mysquare}} \put(4,2){\usebox{\mysquare}} \put(6,2){\usebox{\mysquare}} \put(8,2){\usebox{\mysquare}}  \put(10,2){\usebox{\mysquare}}

\put(0,4){\usebox{\mysquare}} \put(2,4){\usebox{\mysquare}} \put(4,4){\usebox{\mysquare}} \put(6,4){\usebox{\mysquare}} \put(8,4){\usebox{\mysquare}}  \put(10,4){\usebox{\mysquare}}

\put(0,6){\usebox{\mysquare}} \put(2,6){\usebox{\mysquare}} \put(4,6){\usebox{\mysquare}} \put(6,6){\usebox{\mysquare}} \put(8,6){\usebox{\mysquare}}  \put(10,6){\usebox{\mysquare}}

\put(2,8){\usebox{\mysquare}} \put(4,8){\usebox{\mysquare}} \put(6,8){\usebox{\mysquare}} \put(8,8){\usebox{\mysquare}}

\put(1,1){\usebox{\mysquare}} \put(3,1){\usebox{\mysquare}}  \put(5,1){\usebox{\mysquare}} \put(7,1){\usebox{\mysquare}} \put(9,1){\usebox{\mysquare}}

\put(1,3){\usebox{\mysquare}} \put(3,3){\usebox{\mysquare}}  \put(5,3){\usebox{\mysquare}} \put(7,3){\usebox{\mysquare}} \put(9,3){\usebox{\mysquare}}

\put(1,5){\usebox{\mysquare}} \put(3,5){\usebox{\mysquare}}  \put(5,5){\usebox{\mysquare}} \put(7,5){\usebox{\mysquare}} \put(9,5){\usebox{\mysquare}}

\put(1,7){\usebox{\mysquare}} \put(3,7){\usebox{\mysquare}}  \put(5,7){\usebox{\mysquare}} \put(7,7){\usebox{\mysquare}} \put(9,7){\usebox{\mysquare}}

  \put(1,0){\line(1,0){1}}  \put(2,0){\line(1,0){1}}  \put(3,0){\line(1,0){1}}     \put(7,0){\line(1,0){1}}  \put(8,0){\line(1,0){1}}

  \put(2,1){\line(1,0){1}}  \put(3,1){\line(1,0){1}}  \put(4,1){\line(1,0){1}}    \put(5,1){\line(1,0){1}}  \put(6,1){\line(1,0){1}}   \put(9,1){\line(1,0){1}}  \put(10,1){\line(1,0){1}}
  \put(0,2){\line(1,0){1}}  \put(1,2){\line(1,0){1}}   \put(5,2){\line(1,0){1}}  \put(5,2){\line(1,0){1}}  \put(6,2){\line(1,0){1}}  \put(7,2){\line(1,0){1}}  \put(8,2){\line(1,0){1}}  \put(9,2){\line(1,0){1}}  \put(10,2){\line(1,0){1}}
  \put(0,3){\line(1,0){1}}  \put(1,3){\line(1,0){1}}  \put(2,3){\line(1,0){1}}  \put(3,3){\line(1,0){1}}  \put(4,3){\line(1,0){1}}  \put(5,3){\line(1,0){1}}  \put(5,3){\line(1,0){1}}  \put(6,3){\line(1,0){1}}  \put(7,3){\line(1,0){1}}  \put(8,3){\line(1,0){1}}  \put(9,3){\line(1,0){1}}  \put(10,3){\line(1,0){1}}
  \put(0,4){\line(1,0){1}}  \put(1,4){\line(1,0){1}}    \put(3,4){\line(1,0){1}}  \put(4,4){\line(1,0){1}}   \put(8,4){\line(1,0){1}}  \put(9,4){\line(1,0){1}}  \put(10,4){\line(1,0){1}}
    \put(2,5){\line(1,0){1}}  \put(3,5){\line(1,0){1}}  \put(4,5){\line(1,0){1}}  \put(5,5){\line(1,0){1}}    \put(6,5){\line(1,0){1}}  \put(7,5){\line(1,0){1}}  \put(9,5){\line(1,0){1}}  \put(10,5){\line(1,0){1}}
  \put(0,6){\line(1,0){1}}  \put(1,6){\line(1,0){1}}  \put(2,6){\line(1,0){1}}  \put(3,6){\line(1,0){1}}  \put(4,6){\line(1,0){1}}  \put(5,6){\line(1,0){1}}  \put(5,6){\line(1,0){1}}  \put(6,6){\line(1,0){1}}  \put(7,6){\line(1,0){1}}  \put(8,6){\line(1,0){1}}  \put(9,6){\line(1,0){1}}  \put(10,6){\line(1,0){1}}
  \put(0,7){\line(1,0){1}}  \put(1,7){\line(1,0){1}}  \put(2,7){\line(1,0){1}}  \put(3,7){\line(1,0){1}}   \put(5,7){\line(1,0){1}}  \put(6,7){\line(1,0){1}}  \put(8,7){\line(1,0){1}}  \put(9,7){\line(1,0){1}}  \put(10,7){\line(1,0){1}}
    \put(1,8){\line(1,0){1}}  \put(2,8){\line(1,0){1}}    \put(4,8){\line(1,0){1}}  \put(5,8){\line(1,0){1}}  \put(6,8){\line(1,0){1}}  \put(7,8){\line(1,0){1}}    \put(9,8){\line(1,0){1}}  \put(10,8){\line(1,0){1}}
    \put(1,9){\line(1,0){1}}  \put(2,9){\line(1,0){1}}  \put(3,9){\line(1,0){1}}  \put(4,9){\line(1,0){1}}  \put(5,9){\line(1,0){1}}  \put(6,9){\line(1,0){1}}  \put(7,9){\line(1,0){1}}  \put(8,9){\line(1,0){1}}

    \put(1,0){\line(0,1){1}}  \put(2,0){\line(0,1){1}}   \put(4,0){\line(0,1){1}}      \put(7,0){\line(0,1){1}}  \put(8,0){\line(0,1){1}}  \put(9,0){\line(0,1){1}}
  \put(1,1){\line(0,1){1}}  \put(2,1){\line(0,1){1}}  \put(3,1){\line(0,1){1}}  \put(4,1){\line(0,1){1}}  \put(5,1){\line(0,1){1}}  \put(7,1){\line(0,1){1}}  \put(8,1){\line(0,1){1}} \put(9,1){\line(0,1){1}}   \put(11,1){\line(0,1){1}}
  \put(0,2){\line(0,1){1}}   \put(2,2){\line(0,1){1}}  \put(3,2){\line(0,1){1}}  \put(4,2){\line(0,1){1}} \put(5,2){\line(0,1){1}}  \put(7,2){\line(0,1){1}}    \put(9,2){\line(0,1){1}}   \put(11,2){\line(0,1){1}}
  \put(0,3){\line(0,1){1}}  \put(2,3){\line(0,1){1}}  \put(3,3){\line(0,1){1}}  \put(5,3){\line(0,1){1}}  \put(6,3){\line(0,1){1}}  \put(7,3){\line(0,1){1}}  \put(8,3){\line(0,1){1}}  \put(10,3){\line(0,1){1}}
  \put(0,4){\line(0,1){1}}  \put(1,4){\line(0,1){1}}  \put(2,4){\line(0,1){1}}  \put(3,4){\line(0,1){1}}     \put(5,4){\line(0,1){1}}  \put(6,4){\line(0,1){1}}  \put(7,4){\line(0,1){1}}  \put(8,4){\line(0,1){1}}  \put(9,4){\line(0,1){1}}  \put(11,4){\line(0,1){1}}
  \put(0,5){\line(0,1){1}}  \put(1,5){\line(0,1){1}}  \put(2,5){\line(0,1){1}}  \put(4,5){\line(0,1){1}}  \put(6,5){\line(0,1){1}}  \put(8,5){\line(0,1){1}}  \put(9,5){\line(0,1){1}}  \put(11,5){\line(0,1){1}}
  \put(0,6){\line(0,1){1}}  \put(2,6){\line(0,1){1}}  \put(4,6){\line(0,1){1}}    \put(5,6){\line(0,1){1}}  \put(7,6){\line(0,1){1}}  \put(8,6){\line(0,1){1}}  \put(10,6){\line(0,1){1}}
  \put(1,7){\line(0,1){1}}  \put(3,7){\line(0,1){1}}  \put(4,7){\line(0,1){1}}    \put(5,7){\line(0,1){1}}    \put(7,7){\line(0,1){1}}  \put(8,7){\line(0,1){1}}  \put(9,7){\line(0,1){1}}  \put(11,7){\line(0,1){1}}
  \put(1,8){\line(0,1){1}}  \put(3,8){\line(0,1){1}}  \put(4,8){\line(0,1){1}}  \put(6,8){\line(0,1){1}}  \put(8,8){\line(0,1){1}}  \put(9,8){\line(0,1){1}}

\put(0.1,.1){} \put(1.1,.1){0}  \put(2.1,.1){1}  \put(3.1,.1){0}  \put(4.1,.1){2}  \put(5.1,.1){}  \put(6.1,.1){}  \put(7.1,.1){4}  \put(8.1,.1){5}  \put(9.1,.1){4}  \put(10.1,.1){}  \put(11.1,.1){}

\put(0.1,1.1){} \put(1.1,1.1){-1}  \put(2.1,1.1){2}  \put(3.1,1.1){3}  \put(4.1,1.1){2}  \put(5.1,1.1){3}  \put(6.1,1.1){2}  \put(7.1,1.1){3}  \put(8.1,1.1){6}  \put(9.1,1.1){3}  \put(10.1,1.1){2}  \put(11.1,1.1){3}

\put(0.1,2.1){1} \put(1.1,2.1){0}  \put(2.1,2.1){1}  \put(3.1,2.1){4}  \put(4.1,2.1){1}  \put(5.1,2.1){4}  \put(6.1,2.1){5}  \put(7.1,2.1){4}  \put(8.1,2.1){5}  \put(9.1,2.1){4}  \put(10.1,2.1){5}  \put(11.1,2.1){4}

\put(0.1,3.1){2} \put(1.1,3.1){3}  \put(2.1,3.1){2}  \put(3.1,3.1){3}  \put(4.1,3.1){2}  \put(5.1,3.1){3}  \put(6.1,3.1){2}  \put(7.1,3.1){3}  \put(8.1,3.1){2}  \put(9.1,3.1){3}  \put(10.1,3.1){2}  \put(11.1,3.1){3}

\put(0.1,4.1){1} \put(1.1,4.1){0}  \put(2.1,4.1){1}  \put(3.1,4.1){4}  \put(4.1,4.1){5}  \put(5.1,4.1){4}  \put(6.1,4.1){1}  \put(7.1,4.1){4}  \put(8.1,4.1){1}  \put(9.1,4.1){0}  \put(10.1,4.1){1}  \put(11.1,4.1){0}

\put(0.1,5.1){2} \put(1.1,5.1){-1}  \put(2.1,5.1){2}  \put(3.1,5.1){3}  \put(4.1,5.1){2}  \put(5.1,5.1){3}  \put(6.1,5.1){2}  \put(7.1,5.1){3}  \put(8.1,5.1){2}  \put(9.1,5.1){-1}  \put(10.1,5.1){-2}  \put(11.1,5.1){-1}

\put(0.1,6.1){1} \put(1.1,6.1){0}  \put(2.1,6.1){1}  \put(3.1,6.1){0}  \put(4.1,6.1){1}  \put(5.1,6.1){0}  \put(6.1,6.1){1}  \put(7.1,6.1){0}  \put(8.1,6.1){1}  \put(9.1,6.1){0}  \put(10.1,6.1){1}  \put(11.1,6.1){0}

\put(0.1,7.1){2} \put(1.1,7.1){3}  \put(2.1,7.1){2}  \put(3.1,7.1){3}  \put(4.1,7.1){2}  \put(5.1,7.1){-1}  \put(6.1,7.1){-2}  \put(7.1,7.1){-1} \put(8.1,7.1){2}  \put(9.1,7.1){3}  \put(10.1,7.1){2}  \put(11.1,7.1){3}

\put(0.1,8.1){} \put(1.1,8.1){4}  \put(2.1,8.1){5}  \put(3.1,8.1){4}  \put(4.1,8.1){1}  \put(5.1,8.1){0}  \put(6.1,8.1){1}  \put(7.1,8.1){0}  \put(8.1,8.1){1}  \put(9.1,8.1){4}  \put(10.1,8.1){5}  \put(11.1,8.1){4}

\put(0.1,9.1){} \put(1.1,9.1){3}  \put(2.1,9.1){2}  \put(3.1,9.1){3}  \put(4.1,9.1){2}  \put(5.1,9.1){3}  \put(6.1,9.1){2}  \put(7.1,9.1){3}  \put(8.1,9.1){2}  \put(9.1,9.1){3}  \put(10.1,9.1){}  \put(11.1,9.1){}



\end{picture}

\caption[Short Caption]{A domino tiling with its height function. \label{bayh}}
\end{figure}


One useful way to think of a domino tiling as a graph. It has
vertices $(\Z+\frac12)^2$ and an edge between $u$ and $x(u)$ for all
$u$. The second condition of a domino tiling
implies that every vertex in this graph has degree one. With this
interpretation we can consider the union of two domino tilings. This
interpretation will be very useful for studying the height function.

For any domino tiling $x$ and any $N \in \N$ we define
$$ \tilde x_N=\left\{y:\ y|_{\Z^2\setminus [-N,N]^2}=x|_{\Z^2\setminus [-N,N]^2}\right\}.$$
The following theorem of Kenyon shows that cycles of $y \cup y'$ are critical to understanding the
difference in the height functions.

\begin{theorem} \cite{kenyon2} \label{thm:kenyon2}
For all $N$, $x$ and $y,y' \in \tilde x_N$
\begin{enumerate}
\item If there exists a path from $u$ to $v$ which does not cross a cycle of
      $y \cup y'$ then
      $$h_{y}(u)-h_{y'}(u)=h_{y}(v)-h_{y'}(v). $$
\item $h_{y}-h_{y'}$ increases or decreases by 4 every time you cross a cycle of $y \cup y'$ and
\item conditioned on $y,y' \in \tilde x_N$ and the cycles in $y \cup y'$, the increases or decreases of $h_{y}-h_{y'}$
      are mutually independent for all the cycles in $y \cup
      y'$.
\end{enumerate}
\end{theorem}

To illustrate this theorem in Figures \ref{toucan} and \ref{toucan1} we show the previous domino tiling $y$ and another domino tiling $y'$ such that $y\cup y'$ has a cycle.  The height functions for both $x$ and $y$ are shown.  Note that the height function in the second tiling agrees with the height function in the first outside the cycle and is four greater than the height function for the first tiling inside the cycle.


%
\setlength{\unitlength}{12mm}
\savebox{\mysquare}{\textcolor{yellow}{\rule{23mm}{11mm}}}

\newsavebox{\bluesquare}
\savebox{\bluesquare}{\textcolor{blue}{\rule{23mm}{11mm}}}
\begin{figure}
\begin{picture}(12,10)

  \put(1,0){\line(1,0){1}}  \put(2,0){\line(1,0){1}}  \put(3,0){\line(1,0){1}}     \put(7,0){\line(1,0){1}}  \put(8,0){\line(1,0){1}}

%
%
%
%

\put(3.04,1.04){\rotatebox{90}{\usebox{\bluesquare}}}
\put(4.04,1.04){\rotatebox{90}{\usebox{\bluesquare}}}
\put(5.04,2.04){\usebox{\bluesquare}}
\put(7.04,2.04){\usebox{\bluesquare}}

\put(2.04,6.04){\usebox{\bluesquare}}
\put(2.04,5.04){\usebox{\bluesquare}}
\put(3.04,4.04){\usebox{\bluesquare}}
\put(3.04,3.04){\usebox{\bluesquare}}

\put(5.04,7.04){\usebox{\bluesquare}}

\put(8.04,3.04){\rotatebox{00}{\usebox{\bluesquare}}}
\put(9.04,4.04){\rotatebox{00}{\usebox{\bluesquare}}}
\put(9.04,5.04){\rotatebox{00}{\usebox{\bluesquare}}}
\put(8.04,6.04){\rotatebox{00}{\usebox{\bluesquare}}}
\put(7.04,6.04){\rotatebox{90}{\usebox{\bluesquare}}}
\put(4.04,6.04){\rotatebox{90}{\usebox{\bluesquare}}}


  \put(2,1){\line(1,0){1}}  \put(3,1){\line(1,0){1}}  \put(4,1){\line(1,0){1}}    \put(5,1){\line(1,0){1}}  \put(6,1){\line(1,0){1}}   \put(9,1){\line(1,0){1}}  \put(10,1){\line(1,0){1}}
  \put(0,2){\line(1,0){1}}  \put(1,2){\line(1,0){1}}   \put(5,2){\line(1,0){1}}  \put(5,2){\line(1,0){1}}  \put(6,2){\line(1,0){1}}  \put(7,2){\line(1,0){1}}  \put(8,2){\line(1,0){1}}  \put(9,2){\line(1,0){1}}  \put(10,2){\line(1,0){1}}
  \put(0,3){\line(1,0){1}}  \put(1,3){\line(1,0){1}}  \put(2,3){\line(1,0){1}}  \put(4,3){\line(1,0){1}}  \put(5,3){\line(1,0){1}}  \put(5,3){\line(1,0){1}}  \put(6,3){\line(1,0){1}}  \put(7,3){\line(1,0){1}}  \put(9,3){\line(1,0){1}}  \put(10,3){\line(1,0){1}}
  \put(0,4){\line(1,0){1}}  \put(1,4){\line(1,0){1}}  \put(3,4){\line(1,0){1}}   \put(8,4){\line(1,0){1}}  \put(10,4){\line(1,0){1}}
    \put(2,5){\line(1,0){1}}  \put(4,5){\line(1,0){1}}  \put(5,5){\line(1,0){1}}    \put(6,5){\line(1,0){1}}  \put(7,5){\line(1,0){1}}  \put(9,5){\line(1,0){1}}
  \put(0,6){\line(1,0){1}}  \put(1,6){\line(1,0){1}}  \put(3,6){\line(1,0){1}}  \put(4,6){\line(1,0){1}}  \put(5,6){\line(1,0){1}}  \put(5,6){\line(1,0){1}}  \put(6,6){\line(1,0){1}}  \put(7,6){\line(1,0){1}}  \put(8,6){\line(1,0){1}}  \put(10,6){\line(1,0){1}}
  \put(0,7){\line(1,0){1}}  \put(1,7){\line(1,0){1}}  \put(2,7){\line(1,0){1}}  \put(3,7){\line(1,0){1}}   \put(5,7){\line(1,0){1}}  \put(6,7){\line(1,0){1}}  \put(8,7){\line(1,0){1}}  \put(9,7){\line(1,0){1}}  \put(10,7){\line(1,0){1}}
    \put(1,8){\line(1,0){1}}  \put(2,8){\line(1,0){1}}    \put(4,8){\line(1,0){1}}  \put(5,8){\line(1,0){1}}  \put(6,8){\line(1,0){1}}  \put(7,8){\line(1,0){1}}    \put(9,8){\line(1,0){1}}  \put(10,8){\line(1,0){1}}
    \put(1,9){\line(1,0){1}}  \put(2,9){\line(1,0){1}}  \put(3,9){\line(1,0){1}}  \put(4,9){\line(1,0){1}}  \put(5,9){\line(1,0){1}}  \put(6,9){\line(1,0){1}}  \put(7,9){\line(1,0){1}}  \put(8,9){\line(1,0){1}}

    \put(1,0){\line(0,1){1}}  \put(2,0){\line(0,1){1}}   \put(4,0){\line(0,1){1}}      \put(7,0){\line(0,1){1}}  \put(8,0){\line(0,1){1}}  \put(9,0){\line(0,1){1}}
  \put(1,1){\line(0,1){1}}  \put(2,1){\line(0,1){1}}  \put(3,1){\line(0,1){1}}  \put(5,1){\line(0,1){1}}  \put(7,1){\line(0,1){1}}  \put(8,1){\line(0,1){1}} \put(9,1){\line(0,1){1}}   \put(11,1){\line(0,1){1}}
  \put(0,2){\line(0,1){1}}   \put(2,2){\line(0,1){1}}  \put(3,2){\line(0,1){1}}  \put(4,2){\line(0,1){1}} \put(9,2){\line(0,1){1}}   \put(11,2){\line(0,1){1}}
  \put(0,3){\line(0,1){1}}  \put(2,3){\line(0,1){1}}  \put(3,3){\line(0,1){1}}  \put(5,3){\line(0,1){1}}  \put(6,3){\line(0,1){1}}  \put(7,3){\line(0,1){1}}  \put(8,3){\line(0,1){1}}  \put(10,3){\line(0,1){1}}
  \put(0,4){\line(0,1){1}}  \put(1,4){\line(0,1){1}}  \put(2,4){\line(0,1){1}}  \put(3,4){\line(0,1){1}}     \put(5,4){\line(0,1){1}}  \put(6,4){\line(0,1){1}}  \put(7,4){\line(0,1){1}}  \put(8,4){\line(0,1){1}}  \put(9,4){\line(0,1){1}}  \put(11,4){\line(0,1){1}}
  \put(0,5){\line(0,1){1}}  \put(1,5){\line(0,1){1}}  \put(2,5){\line(0,1){1}}  \put(4,5){\line(0,1){1}}  \put(6,5){\line(0,1){1}}  \put(8,5){\line(0,1){1}}  \put(9,5){\line(0,1){1}}  \put(11,5){\line(0,1){1}}
  \put(0,6){\line(0,1){1}}  \put(2,6){\line(0,1){1}}  \put(5,6){\line(0,1){1}}  \put(7,6){\line(0,1){1}}    \put(10,6){\line(0,1){1}}
  \put(1,7){\line(0,1){1}}  \put(3,7){\line(0,1){1}}  \put(4,7){\line(0,1){1}}    \put(8,7){\line(0,1){1}}  \put(9,7){\line(0,1){1}}  \put(11,7){\line(0,1){1}}
  \put(1,8){\line(0,1){1}}  \put(3,8){\line(0,1){1}}  \put(4,8){\line(0,1){1}}  \put(6,8){\line(0,1){1}}  \put(8,8){\line(0,1){1}}  \put(9,8){\line(0,1){1}}

\put(0.1,.1){} \put(1.1,.1){0}  \put(2.1,.1){1}  \put(3.1,.1){0}  \put(4.1,.1){2}  \put(5.1,.1){}  \put(6.1,.1){}  \put(7.1,.1){4}  \put(8.1,.1){5}  \put(9.1,.1){4}  \put(10.1,.1){}  \put(11.1,.1){}

\put(0.1,1.1){} \put(1.1,1.1){-1}  \put(2.1,1.1){2}  \put(3.1,1.1){3}  \put(4.1,1.1){2}  \put(5.1,1.1){3}  \put(6.1,1.1){2}  \put(7.1,1.1){3}  \put(8.1,1.1){6}  \put(9.1,1.1){3}  \put(10.1,1.1){2}  \put(11.1,1.1){3}

\put(0.1,2.1){1} \put(1.1,2.1){0}  \put(2.1,2.1){1}  \put(3.1,2.1){4}  \put(4.1,2.1){1}  \put(5.1,2.1){4}  \put(6.1,2.1){5}  \put(7.1,2.1){4}  \put(8.1,2.1){5}  \put(9.1,2.1){4}  \put(10.1,2.1){5}  \put(11.1,2.1){4}

\put(0.1,3.1){2} \put(1.1,3.1){3}  \put(2.1,3.1){2}  \put(3.1,3.1){3}  \put(4.1,3.1){2}  \put(5.1,3.1){3}  \put(6.1,3.1){2}  \put(7.1,3.1){3}  \put(8.1,3.1){2}  \put(9.1,3.1){3}  \put(10.1,3.1){2}  \put(11.1,3.1){3}

\put(0.1,4.1){1} \put(1.1,4.1){0}  \put(2.1,4.1){1}  \put(3.1,4.1){4}  \put(4.1,4.1){5}  \put(5.1,4.1){4}  \put(6.1,4.1){1}  \put(7.1,4.1){4}  \put(8.1,4.1){1}  \put(9.1,4.1){0}  \put(10.1,4.1){1}  \put(11.1,4.1){0}

\put(0.1,5.1){2} \put(1.1,5.1){-1}  \put(2.1,5.1){2}  \put(3.1,5.1){3}  \put(4.1,5.1){2}  \put(5.1,5.1){3}  \put(6.1,5.1){2}  \put(7.1,5.1){3}  \put(8.1,5.1){2}  \put(9.1,5.1){-1}  \put(10.1,5.1){-2}  \put(11.1,5.1){-1}

\put(0.1,6.1){1} \put(1.1,6.1){0}  \put(2.1,6.1){1}  \put(3.1,6.1){0}  \put(4.1,6.1){1}  \put(5.1,6.1){0}  \put(6.1,6.1){1}  \put(7.1,6.1){0}  \put(8.1,6.1){1}  \put(9.1,6.1){0}  \put(10.1,6.1){1}  \put(11.1,6.1){0}

\put(0.1,7.1){2} \put(1.1,7.1){3}  \put(2.1,7.1){2}  \put(3.1,7.1){3}  \put(4.1,7.1){2}  \put(5.1,7.1){-1}  \put(6.1,7.1){-2}  \put(7.1,7.1){-1} \put(8.1,7.1){2}  \put(9.1,7.1){3}  \put(10.1,7.1){2}  \put(11.1,7.1){3}

\put(0.1,8.1){} \put(1.1,8.1){4}  \put(2.1,8.1){5}  \put(3.1,8.1){4}  \put(4.1,8.1){1}  \put(5.1,8.1){0}  \put(6.1,8.1){1}  \put(7.1,8.1){0}  \put(8.1,8.1){1}  \put(9.1,8.1){4}  \put(10.1,8.1){5}  \put(11.1,8.1){4}

\put(0.1,9.1){} \put(1.1,9.1){3}  \put(2.1,9.1){2}  \put(3.1,9.1){3}  \put(4.1,9.1){2}  \put(5.1,9.1){3}  \put(6.1,9.1){2}  \put(7.1,9.1){3}  \put(8.1,9.1){2}  \put(9.1,9.1){3}  \put(10.1,9.1){}  \put(11.1,9.1){}

  \linethickness{0.75mm}
  \put(3.5,1.5){\line(1,0){1}}
  \put(4.5,2.5){\line(1,0){4}}
  \put(8.5,3.5){\line(1,0){1}}
  \put(9.5,4.5){\line(1,0){1}}
  \put(9.5,5.5){\line(1,0){1}}
  \put(7.5,6.5){\line(1,0){2}}
  \put(4.5,7.5){\line(1,0){3}}
  \put(2.5,6.5){\line(1,0){2}}
  \put(2.5,5.5){\line(1,0){1}}
  \put(3.5,4.5){\line(1,0){1}}
  \put(3.5,3.5){\line(1,0){1}}

  \put(4.5,1.5){\line(0,1){1}}
  \put(8.5,2.5){\line(0,1){1}}
  \put(9.5,3.5){\line(0,1){1}}
  \put(9.5,5.5){\line(0,1){1}}
  \put(10.5,4.5){\line(0,1){1}}
  \put(7.5,6.5){\line(0,1){1}}
  \put(4.5,6.5){\line(0,1){1}}
  \put(2.5,5.5){\line(0,1){1}}
  \put(3.5,4.5){\line(0,1){1}}
  \put(4.5,3.5){\line(0,1){1}}
  \put(3.5,1.5){\line(0,1){2}}


\end{picture}
 \caption{A portion of a domino tiling $y$ with a cycle (in $y \cup y'$, where $y'$ is on the next page) highlighted. \label{toucan} }

\end{figure}

\begin{figure}
\begin{picture}(12,10) \label{house}
  \put(1,0){\line(1,0){1}}  \put(2,0){\line(1,0){1}}  \put(3,0){\line(1,0){1}}     \put(7,0){\line(1,0){1}}  \put(8,0){\line(1,0){1}}

\put(3.04,1.04){\usebox{\mysquare}}
\put(4.04,2.04){\usebox{\mysquare}}
\put(6.04,2.04){\usebox{\mysquare}}

\put(3.04,6.04){\usebox{\mysquare}}
\put(4.04,7.04){\usebox{\mysquare}}
\put(6.04,7.04){\usebox{\mysquare}}
\put(7.04,6.04){\usebox{\mysquare}}

\put(8.04,2.04){\rotatebox{90}{\usebox{\mysquare}}}
\put(9.04,3.04){\rotatebox{90}{\usebox{\mysquare}}}
\put(10.04,4.04){\rotatebox{90}{\usebox{\mysquare}}}
\put(9.04,5.04){\rotatebox{90}{\usebox{\mysquare}}}

\put(3.04,2.04){\rotatebox{90}{\usebox{\mysquare}}}
\put(4.04,3.04){\rotatebox{90}{\usebox{\mysquare}}}
\put(3.04,4.04){\rotatebox{90}{\usebox{\mysquare}}}
\put(2.04,5.04){\rotatebox{90}{\usebox{\mysquare}}}



  \put(2,1){\line(1,0){1}}  \put(3,1){\line(1,0){1}}  \put(4,1){\line(1,0){1}}    \put(5,1){\line(1,0){1}}  \put(6,1){\line(1,0){1}}   \put(9,1){\line(1,0){1}}  \put(10,1){\line(1,0){1}}
  \put(0,2){\line(1,0){1}}  \put(1,2){\line(1,0){1}}   \put(5,2){\line(1,0){1}}  \put(5,2){\line(1,0){1}}  \put(6,2){\line(1,0){1}}  \put(7,2){\line(1,0){1}}  \put(8,2){\line(1,0){1}}  \put(9,2){\line(1,0){1}}  \put(10,2){\line(1,0){1}}
  \put(0,3){\line(1,0){1}}  \put(1,3){\line(1,0){1}}  \put(2,3){\line(1,0){1}}  \put(4,3){\line(1,0){1}}  \put(5,3){\line(1,0){1}}  \put(5,3){\line(1,0){1}}  \put(6,3){\line(1,0){1}}  \put(7,3){\line(1,0){1}}  \put(9,3){\line(1,0){1}}  \put(10,3){\line(1,0){1}}
  \put(0,4){\line(1,0){1}}  \put(1,4){\line(1,0){1}}  \put(3,4){\line(1,0){1}}   \put(8,4){\line(1,0){1}}  \put(10,4){\line(1,0){1}}
    \put(2,5){\line(1,0){1}}  \put(4,5){\line(1,0){1}}  \put(5,5){\line(1,0){1}}    \put(6,5){\line(1,0){1}}  \put(7,5){\line(1,0){1}}  \put(9,5){\line(1,0){1}}
  \put(0,6){\line(1,0){1}}  \put(1,6){\line(1,0){1}}  \put(3,6){\line(1,0){1}}  \put(4,6){\line(1,0){1}}  \put(5,6){\line(1,0){1}}  \put(5,6){\line(1,0){1}}  \put(6,6){\line(1,0){1}}  \put(7,6){\line(1,0){1}}  \put(8,6){\line(1,0){1}}  \put(10,6){\line(1,0){1}}
  \put(0,7){\line(1,0){1}}  \put(1,7){\line(1,0){1}}  \put(2,7){\line(1,0){1}}  \put(3,7){\line(1,0){1}}   \put(5,7){\line(1,0){1}}  \put(6,7){\line(1,0){1}}  \put(8,7){\line(1,0){1}}  \put(9,7){\line(1,0){1}}  \put(10,7){\line(1,0){1}}
    \put(1,8){\line(1,0){1}}  \put(2,8){\line(1,0){1}}    \put(4,8){\line(1,0){1}}  \put(5,8){\line(1,0){1}}  \put(6,8){\line(1,0){1}}  \put(7,8){\line(1,0){1}}    \put(9,8){\line(1,0){1}}  \put(10,8){\line(1,0){1}}
    \put(1,9){\line(1,0){1}}  \put(2,9){\line(1,0){1}}  \put(3,9){\line(1,0){1}}  \put(4,9){\line(1,0){1}}  \put(5,9){\line(1,0){1}}  \put(6,9){\line(1,0){1}}  \put(7,9){\line(1,0){1}}  \put(8,9){\line(1,0){1}}

    \put(1,0){\line(0,1){1}}  \put(2,0){\line(0,1){1}}   \put(4,0){\line(0,1){1}}      \put(7,0){\line(0,1){1}}  \put(8,0){\line(0,1){1}}  \put(9,0){\line(0,1){1}}
  \put(1,1){\line(0,1){1}}  \put(2,1){\line(0,1){1}}  \put(3,1){\line(0,1){1}}  \put(5,1){\line(0,1){1}}  \put(7,1){\line(0,1){1}}  \put(8,1){\line(0,1){1}} \put(9,1){\line(0,1){1}}   \put(11,1){\line(0,1){1}}
  \put(0,2){\line(0,1){1}}   \put(2,2){\line(0,1){1}}  \put(3,2){\line(0,1){1}}  \put(4,2){\line(0,1){1}} \put(9,2){\line(0,1){1}}   \put(11,2){\line(0,1){1}}
  \put(0,3){\line(0,1){1}}  \put(2,3){\line(0,1){1}}  \put(3,3){\line(0,1){1}}  \put(5,3){\line(0,1){1}}  \put(6,3){\line(0,1){1}}  \put(7,3){\line(0,1){1}}  \put(8,3){\line(0,1){1}}  \put(10,3){\line(0,1){1}}
  \put(0,4){\line(0,1){1}}  \put(1,4){\line(0,1){1}}  \put(2,4){\line(0,1){1}}  \put(3,4){\line(0,1){1}}     \put(5,4){\line(0,1){1}}  \put(6,4){\line(0,1){1}}  \put(7,4){\line(0,1){1}}  \put(8,4){\line(0,1){1}}  \put(9,4){\line(0,1){1}}  \put(11,4){\line(0,1){1}}
  \put(0,5){\line(0,1){1}}  \put(1,5){\line(0,1){1}}  \put(2,5){\line(0,1){1}}  \put(4,5){\line(0,1){1}}  \put(6,5){\line(0,1){1}}  \put(8,5){\line(0,1){1}}  \put(9,5){\line(0,1){1}}  \put(11,5){\line(0,1){1}}
  \put(0,6){\line(0,1){1}}  \put(2,6){\line(0,1){1}}  \put(5,6){\line(0,1){1}}  \put(7,6){\line(0,1){1}}    \put(10,6){\line(0,1){1}}
  \put(1,7){\line(0,1){1}}  \put(3,7){\line(0,1){1}}  \put(4,7){\line(0,1){1}}    \put(8,7){\line(0,1){1}}  \put(9,7){\line(0,1){1}}  \put(11,7){\line(0,1){1}}
  \put(1,8){\line(0,1){1}}  \put(3,8){\line(0,1){1}}  \put(4,8){\line(0,1){1}}  \put(6,8){\line(0,1){1}}  \put(8,8){\line(0,1){1}}  \put(9,8){\line(0,1){1}}

\put(0.1,.1){} \put(1.1,.1){0}  \put(2.1,.1){1}  \put(3.1,.1){0}  \put(4.1,.1){2}  \put(5.1,.1){}  \put(6.1,.1){}  \put(7.1,.1){4}  \put(8.1,.1){5}  \put(9.1,.1){4}  \put(10.1,.1){}  \put(11.1,.1){}

\put(0.1,1.1){} \put(1.1,1.1){-1}  \put(2.1,1.1){2}  \put(3.1,1.1){3}  \put(4.1,1.1){2}  \put(5.1,1.1){3}  \put(6.1,1.1){2}  \put(7.1,1.1){3}  \put(8.1,1.1){6}  \put(9.1,1.1){3}  \put(10.1,1.1){2}  \put(11.1,1.1){3}

\put(0.1,2.1){1} \put(1.1,2.1){0}  \put(2.1,2.1){1}  \put(3.1,2.1){4}  \put(4.1,2.1){5}  \put(5.1,2.1){4}  \put(6.1,2.1){5}  \put(7.1,2.1){4}  \put(8.1,2.1){5}  \put(9.1,2.1){4}  \put(10.1,2.1){5}  \put(11.1,2.1){4}

\put(0.1,3.1){2} \put(1.1,3.1){3}  \put(2.1,3.1){2}  \put(3.1,3.1){3}  \put(4.1,3.1){6}  \put(5.1,3.1){7}  \put(6.1,3.1){6}  \put(7.1,3.1){7}  \put(8.1,3.1){6}  \put(9.1,3.1){3}  \put(10.1,3.1){2}  \put(11.1,3.1){3}

\put(0.1,4.1){1} \put(1.1,4.1){0}  \put(2.1,4.1){1}  \put(3.1,4.1){4}  \put(4.1,4.1){5}  \put(5.1,4.1){8}  \put(6.1,4.1){5}  \put(7.1,4.1){8}  \put(8.1,4.1){5}  \put(9.1,4.1){4}  \put(10.1,4.1){1}  \put(11.1,4.1){0}

\put(0.1,5.1){2} \put(1.1,5.1){-1}  \put(2.1,5.1){2}  \put(3.1,5.1){3}  \put(4.1,5.1){6}  \put(5.1,5.1){7}  \put(6.1,5.1){6}  \put(7.1,5.1){7}  \put(8.1,5.1){6}  \put(9.1,5.1){3}  \put(10.1,5.1){2}  \put(11.1,5.1){-1}

\put(0.1,6.1){1} \put(1.1,6.1){0}  \put(2.1,6.1){1}  \put(3.1,6.1){4}  \put(4.1,6.1){5}  \put(5.1,6.1){4}  \put(6.1,6.1){5}  \put(7.1,6.1){4}  \put(8.1,6.1){5}  \put(9.1,6.1){4}  \put(10.1,6.1){1}  \put(11.1,6.1){0}

\put(0.1,7.1){2} \put(1.1,7.1){3}  \put(2.1,7.1){2}  \put(3.1,7.1){3}  \put(4.1,7.1){2}  \put(5.1,7.1){-1}  \put(6.1,7.1){-2}  \put(7.1,7.1){-1} \put(8.1,7.1){2}  \put(9.1,7.1){3}  \put(10.1,7.1){2}  \put(11.1,7.1){3}

\put(0.1,8.1){} \put(1.1,8.1){4}  \put(2.1,8.1){5}  \put(3.1,8.1){4}  \put(4.1,8.1){1}  \put(5.1,8.1){0}  \put(6.1,8.1){1}  \put(7.1,8.1){0}  \put(8.1,8.1){1}  \put(9.1,8.1){4}  \put(10.1,8.1){5}  \put(11.1,8.1){4}

\put(0.1,9.1){} \put(1.1,9.1){3}  \put(2.1,9.1){2}  \put(3.1,9.1){3}  \put(4.1,9.1){2}  \put(5.1,9.1){3}  \put(6.1,9.1){2}  \put(7.1,9.1){3}  \put(8.1,9.1){2}  \put(9.1,9.1){3}  \put(10.1,9.1){}  \put(11.1,9.1){}

  \linethickness{0.75mm}
  \put(3.5,1.5){\line(1,0){1}}
  \put(4.5,2.5){\line(1,0){4}}
  \put(8.5,3.5){\line(1,0){1}}
  \put(9.5,4.5){\line(1,0){1}}
  \put(9.5,5.5){\line(1,0){1}}
  \put(7.5,6.5){\line(1,0){2}}
  \put(4.5,7.5){\line(1,0){3}}
  \put(2.5,6.5){\line(1,0){2}}
  \put(2.5,5.5){\line(1,0){1}}
  \put(3.5,4.5){\line(1,0){1}}
  \put(3.5,3.5){\line(1,0){1}}

  \put(4.5,1.5){\line(0,1){1}}
  \put(8.5,2.5){\line(0,1){1}}
  \put(9.5,3.5){\line(0,1){1}}
  \put(9.5,5.5){\line(0,1){1}}
  \put(10.5,4.5){\line(0,1){1}}
  \put(7.5,6.5){\line(0,1){1}}
  \put(4.5,6.5){\line(0,1){1}}
  \put(2.5,5.5){\line(0,1){1}}
  \put(3.5,4.5){\line(0,1){1}}
  \put(4.5,3.5){\line(0,1){1}}
  \put(3.5,1.5){\line(0,1){2}}


\end{picture}

\caption[Short Caption]{A portion of the  domino tiling $y'$.  Note that the height functions for $y$ and $y'$ differ by 4 inside the cycle. \label{try} } \label{toucan1}

\end{figure}
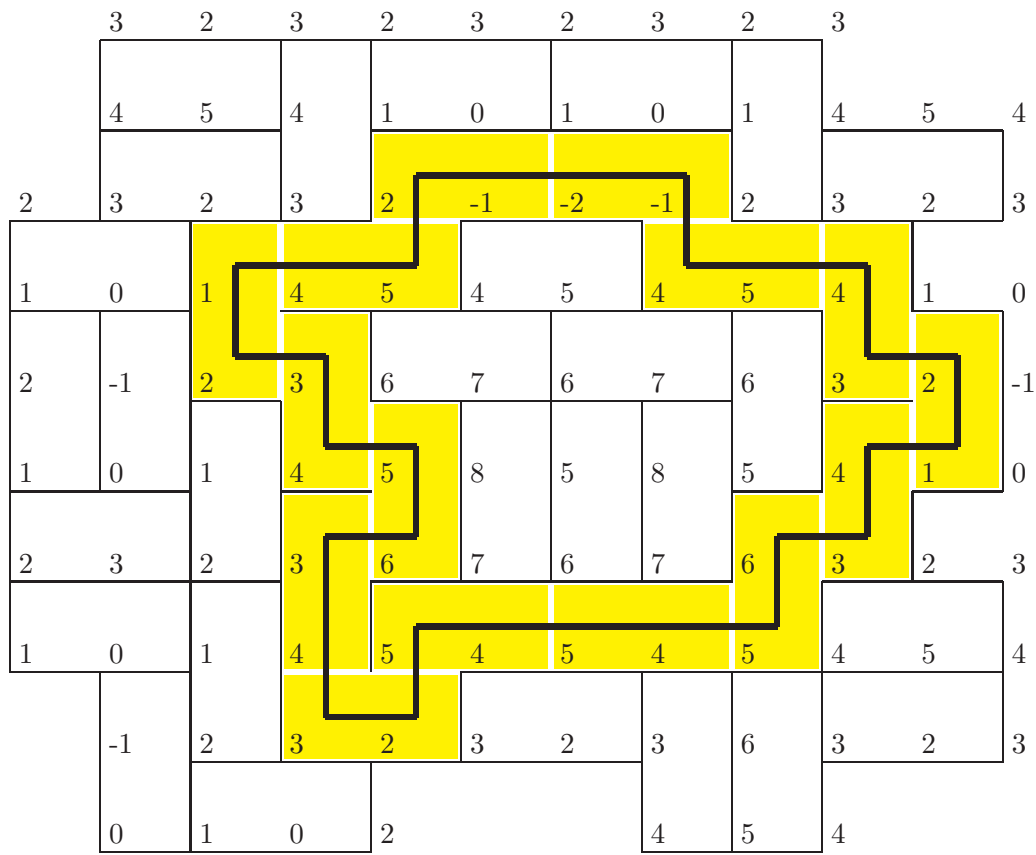
\newpage

%

In \cite{kenyon2} Kenyon proved that asymptotically  the
fluctuations in the height function are conformally invariant.  The
precise version of this theorem that we use is as follows.

Define $\sq_l' \subset \R^2$ be the boundary of the square with vertices at $(\pm l,
\pm l)$ and let $\sq_l=\Z^2 \cap \sq_l'$. Define the annulus $\ann_l$ to be
the region between $\sq_l'$ and $\sq_{2l}'$.

\begin{lem} \cite{kenyon2} \label{cor:kenyon}
There exists $\delta>0$, $N_0$ and $p>0$ such that for all $N>N_0$ and all height functions $H$ and $H'$ with
 $$sup_{v \in \sq_{N} \cup \sq_{2N}}|H(v)|,|H'(v)|<N\delta$$
if we  let $E$ be the event that
 $$h_{x'}|_{\sq_N \cup \sq_{2N}}=H|_{\sq_N \cup \sq_{2N}}$$ and  $$h_{y'}|_{\sq_N \cup \sq_{2N}}=H'|_{\sq_N \cup \sq_{2N}}$$ then
$$\prob\times\prob\bigg((x',y'):\ \exists \text{ a cycle of $x' \cup y'$ in  $\ann_N$} \bigg| \ E \bigg)>p.$$
\end{lem}

\begin{proof}  This version of the conformal invariance for height functions is stated in the
discussion after Theorem 1 in \cite{kenyon2}.
\end{proof}

\section{The colored domino process}\label{construction}
The height function $h_x$ is defined on $\Z ^2$ so
we can easily extend $h_x$ by linearity to the wireframe,
 $(\Z \times \R) \cup (\R \times \Z).$
 Then we define
 \begin{equation} g_x(i,j)=h_x(i,j)+4j \label{spring}\end{equation} on
the wireframe.

Let $$\halfz=\{v \in (\Z \times \R) \cup (\R \times \Z):\  g_x(v)\in
\Z\}.$$

A colored domino tiling of the plane consists of a domino tiling of
the plane $x$ and a function $C: \halfz \to \{1,2\}$ which satisfies
the following coloring rule.
\begin{definition} \label{vancouver}
A map $$C:\halfz \to \{1,2\}$$ satisfies the {\bf coloring rule} if for every $u,v \in \halfz$ with $||u-v||<10$
and $g_x(u)=g_x(v)$ satisfies $C(u)=C(v)$.
\end{definition}

The coloring rule is a global rule as it applies to all $u,v \in \halfz.$ In Section \ref{same} we will show that we can define a ``local coloring rule," that says if $u$ and $v$ are sufficiently close and
$g_x(u)=g_x(v)$ then $C(u)=C(v)$ which implies the coloring rule.  The fact that there is a local coloring rule which is equivalent to Definition \ref{vancouver}
will allow us to show the space
$\dominotilings^*$ is isomorphic to the state space of a subshift of finite type.

Let
$$\dominotilings^*=\bigg\{(x,c):\ x\in \dominotilings \text{ and } c \in \{1,2\}^{\Z} \bigg\}.$$
Let $\sigma(c)_i=c_{i+1}$.
Let $\colorprob$ be the Bernoulli (1/2,1/2) measure on $\{1,2\}^{\Z}$.
We consider the measure $\mu^*=\mu \times \colorprob$ on $\dominotilings^*$.

There are two natural actions $T_{\text{left}}^*$ and $T_{\text{down}}^*$ on $\dominotilings^*$  that preserve $\mu^*$.  One is given by
 $$T_{\text{left}}^*(x,c)=(T_{\text{left}}^2(x),\sigma^{g_x(2,0)}c)$$ and the other
 $$T_{\text{down}}^*(x,c)=(T_{\text{down}}^2(x),\sigma^{g_x(0,2)}c)$$ is defined in the analogous manner.

\begin{remark} \label{whytwo}
The definition of $h_x$ depends on the checkerboard coloring of
$\R^2$.  The checkerboard coloring is not invariant under the shifts
$T_{\text{left}}$ and $T_{\text{down}}$, but is invariant under shifts $T_{\text{left}}^2$ and $T_{\text{down}}^2$.
Because of this
$$h_{T_{\text{left}}^2(x)}(v)=h_x(v+(2,0))-h_x(2,0)$$
but the
relationship between $h_{T_{\text{left}}(x)}$ and $h_x$ is not so simple. Also
$g_x$ and $\halfz$ do not behave nicely under $T_{\text{left}}$ and $T_{\text{down}}$ but do
behave well under $T_{\text{left}}^2$ and $T_{\text{down}}^2$. If we tried to construct a subshift of finite type
which has an isomorphic state space and the shifts $T_{\text{left}}$ and $T_{\text{down}}$, then we would need different tiles for the white squares and the black squares in the underlying checkerboard coloring of the plane.  This would result in a two point factor.
\end{remark}

We conclude this section by proving some facts about $g_x$ which will be useful (in Section \ref{same}) to show that the coloring rule can be generated by a local coloring rule.

\begin{lemma} \label{local}
For any $x \in \dominotilings$ and $a,b,e,f \in \Z$
$$[a,b] \times [e,f] \cap \halfz$$ is determined by $x\cap [a,b]\times[e,f]$
\end{lemma}
\begin{proof}
For $z \in \Z^2$ we have that $g_x(z) \in \Z$. For $u,v \in [a,b] \times [e,f]$
$$g_x(u)-g_x(v)$$ is determined by $x|_{[a,b] \times [e,f]}.$  Combining these two proves the lemma.
\end{proof}

\begin{lemma} \label{ne} For all $x \in \dominotilings$ and $(i,j) \in \Z ^2$
$$g_x(i,j)+7 \geq g_x(i,j+1)\geq g_x(i,j)+1.$$  Thus for each $x \in \dominotilings$ and  $i,k \in \Z$
there exists a unique $j \in \R$ such that $g_x(i,j)=k$.
\end{lemma}

\begin{proof}
Across any edge in the lattice $h_x$ can change by at most 3 so
$$g_x(i,j+1) -g_x(i,j) = h_x(i,j+1) +4(j+1)-h_x(i,j)-4j \geq 4-3=1$$ and
$$g_x(i,j+1) -g_x(i,j) = h_x(i,j+1) +4(j+1)-h_x(i,j)-4j \leq 4-(-3)=7.$$
\end{proof}

A similar argument gives us the following.

\begin{lemma} \label{nbynw}
For all $x \in \dominotilings$ and $(i,j) \in \Z ^2$
$$g_x(i+1,j-2) \leq g_x(i,j) \leq g_x(i+1,j+2).$$
\end{lemma}

\begin{proof}
For any $i,j \in \Z^2$
$$|h_x(i,j)-h_x(i+1,j)| \leq 3$$
and
$$|h_x(i+1,j-2)-h_x(i+1,j)|,|h_x(i+1,j+2)-h_x(i+1,j)| \leq 4$$
The last equation is true because you can't move in a straight line past two squares and have a white square on your left (or right) both times.  Thus $h_x$ can not increase (or decrease) by 3 on two consecutive edges and can only increase or decrease by 4 over any segment of length two.
Thus
$$|h_x(i+1,j-2)-h_x(i,j)|,|h_x(i,j+2)-h_x(i+1,j)| \leq 7$$
and
$$g_x(i,j) \in (g_x(i+1,j-2),g_x(i+1,j+2)).$$
\end{proof}

Another version of this principle is the following.

\begin{cor} \label{ten}
For any $x \in \dominotilings$ and $-n \leq i,j \leq n$
$$|g_x(i,j)|\leq 10n.$$
\end{cor}
\begin{proof}
Both $h_x$ and $g_x$ can change by at most 3 across any horizontal
edge and by Lemma \ref{ne} $g_x$ can change by at most 7 across any
vertical edge. As $g_x(0,0)=0$ this proves the lemma.
\end{proof}

\section{The subshift of finite type} \label{subshift}
First we describe the set of tiles that we will use in our subshift of finite type.  Let $\tiles$ be the set of all possible ways to tile $[0,2]\times[0,2]$ with dominoes whose corners are at points in $\Z^2$.  We show a representation of the set $\tiles$ in Figure \ref{tilespic}.  By Lemma \ref{local} for each element $D \in \tiles$ we can define $h_D:\{0,1,2\}^2\to \Z$ in the natural way  and extend it by linearity to $g_D$ where
$$g_D:\big(\{0,1,2\}\times [0,2]\big) \ \bigcup \ \big([0,2] \times \{0,1,2\}\big)\to \R.$$
Then let $\halfD$ be the set of $(x,y)$ in
$$\big(\{0,1,2\}\times [0,2]\big) \ \bigcup \   \big([0,2] \times \{0,1,2\}\big)$$ such that
$g_D(x,y) \in \Z$.  Define
$$\colorsD =\text{Range}(g_D) \cap \Z.$$

\begin{definition}
The {\bf alphabet} $A$ for the subshift of finite type is
$$A=\big \{(D,c): \ D \in \tiles, c \in \{1,2\}^{\colorsD} \big\}.$$
\end{definition}

We can extend the map $c \in \{1,2\}^{\colorsD}$ to a map $$\tilde c: \halfD \to \{1,2\}$$ by setting
$\tilde c(v)=c(g(v))$ for all $v \in \halfD$.

\begin{definition}
The {\bf adjacency rule} for the subshift of finite type is that two tiles $(D,c)$ and $(D',c')$ may be adjacent if the maps $$\tilde c:\halfD \to \{1,2\} \text{ and } \tilde c':\halfDprime \to \{1,2\}$$
agree on the portions of the boundaries (of the 2 by 2 tiles) that overlap.
\end{definition}

\begin{definition}
We let the state space $S \subset A^{\Z^2}$ be the set of all points that satisfy the adjacency rule.
\end{definition}

As we will see in Lemma \ref{easter} the adjacency rule  forces a complete domino tiling of the plane.
 We illustrate the alphabet $A$ and the adjacency rule with the following figures. We show a representation of the set \tiles $\mbox{}$ in Figure \ref{tilespic}.   Then we show two elements in \tiles where each point in $\halfD$ is labeled by the value of $g_D$ at that point.  Then for a specific $c \in \{1,2\}^{\colorsD}$ we show its representation as a map from $\halfD \to \{1,2\}$, where 1 is represented by a circle and 2 is represented by a disk.


\setlength{\unitlength}{6mm}

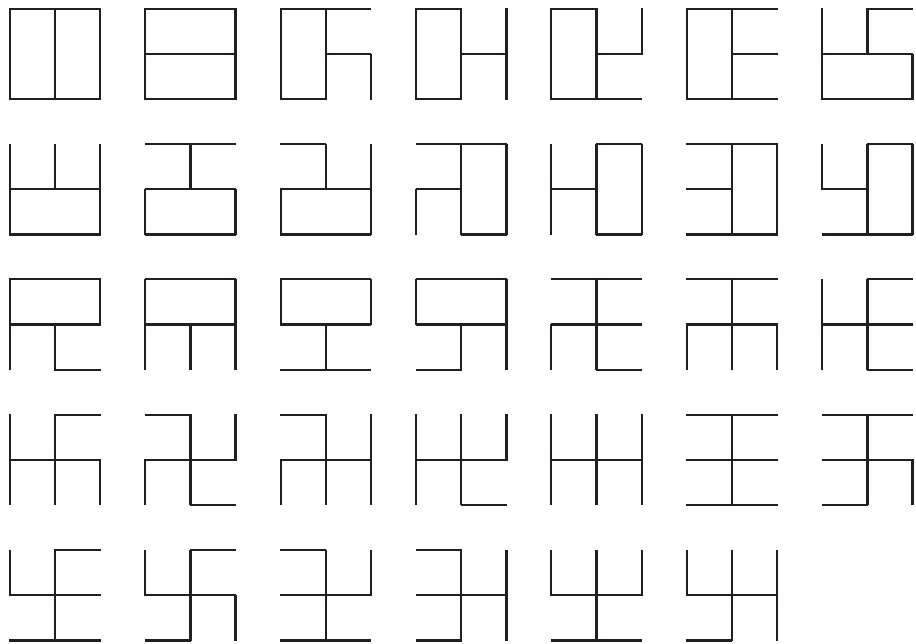
\begin{figure}
\begin{picture}(24,20)

  \put(0,12){\line(1,0){1}}  \put(1,12){\line(1,0){1}}    \put(0,14){\line(1,0){1}}  \put(1,14){\line(1,0){1}}  \put(0,12){\line(0,1){1}}    \put(1,12){\line(0,1){1}}  \put(2,12){\line(0,1){1}}  \put(0,13){\line(0,1){1}}  \put(1,13){\line(0,1){1}}    \put(2,13){\line(0,1){1}}

  \put(3,12){\line(1,0){1}}  \put(4,12){\line(1,0){1}}  \put(3,13){\line(1,0){1}}    \put(4,13){\line(1,0){1}}  \put(3,14){\line(1,0){1}}  \put(4,14){\line(1,0){1}}  \put(3,12){\line(0,1){1}}     \put(5,12){\line(0,1){1}}  \put(3,13){\line(0,1){1}}   \put(5,13){\line(0,1){1}}

  \put(6,12){\line(1,0){1}}        \put(7,13){\line(1,0){1}}  \put(6,14){\line(1,0){1}}  \put(7,14){\line(1,0){1}}  \put(6,12){\line(0,1){1}}    \put(7,12){\line(0,1){1}}  \put(8,12){\line(0,1){1}}  \put(6,13){\line(0,1){1}}  \put(7,13){\line(0,1){1}}

  \put(9,12){\line(1,0){1}}        \put(10,13){\line(1,0){1}}  \put(9,14){\line(1,0){1}}   \put(9,12){\line(0,1){1}}    \put(10,12){\line(0,1){1}}  \put(11,12){\line(0,1){1}}  \put(9,13){\line(0,1){1}}  \put(10,13){\line(0,1){1}}    \put(11,13){\line(0,1){1}}

  \put(12,12){\line(1,0){1}}  \put(13,12){\line(1,0){1}}     \put(13,13){\line(1,0){1}}  \put(12,14){\line(1,0){1}}  \put(12,12){\line(0,1){1}}    \put(13,12){\line(0,1){1}}    \put(12,13){\line(0,1){1}}  \put(13,13){\line(0,1){1}}    \put(14,13){\line(0,1){1}}

  \put(15,12){\line(1,0){1}}  \put(16,12){\line(1,0){1}}      \put(16,13){\line(1,0){1}}  \put(15,14){\line(1,0){1}}  \put(16,14){\line(1,0){1}}  \put(15,12){\line(0,1){1}}    \put(16,12){\line(0,1){1}}    \put(15,13){\line(0,1){1}}  \put(16,13){\line(0,1){1}}

  \put(18,12){\line(1,0){1}}  \put(19,12){\line(1,0){1}}  \put(18,13){\line(1,0){1}}    \put(19,13){\line(1,0){1}}  \put(19,14){\line(1,0){1}}  \put(18,12){\line(0,1){1}}    \put(20,12){\line(0,1){1}}  \put(18,13){\line(0,1){1}}  \put(19,13){\line(0,1){1}}


  \put(0,9){\line(1,0){1}}  \put(1,9){\line(1,0){1}}  \put(0,10){\line(1,0){1}}    \put(1,10){\line(1,0){1}}  \put(0,9){\line(0,1){1}}    \put(2,9){\line(0,1){1}}  \put(0,10){\line(0,1){1}}  \put(1,10){\line(0,1){1}}    \put(2,10){\line(0,1){1}}

  \put(3,9){\line(1,0){1}}  \put(4,9){\line(1,0){1}}  \put(3,10){\line(1,0){1}}    \put(4,10){\line(1,0){1}}  \put(3,11){\line(1,0){1}}  \put(4,11){\line(1,0){1}}  \put(3,9){\line(0,1){1}}    \put(5,9){\line(0,1){1}}  \put(4,10){\line(0,1){1}}

  \put(6,9){\line(1,0){1}}  \put(7,9){\line(1,0){1}}  \put(6,10){\line(1,0){1}}    \put(7,10){\line(1,0){1}}  \put(6,11){\line(1,0){1}}   \put(6,9){\line(0,1){1}}    \put(8,9){\line(0,1){1}} \put(7,10){\line(0,1){1}}    \put(8,10){\line(0,1){1}}

  \put(10,9){\line(1,0){1}}  \put(9,10){\line(1,0){1}}     \put(9,11){\line(1,0){1}}  \put(10,11){\line(1,0){1}}  \put(9,9){\line(0,1){1}}    \put(10,9){\line(0,1){1}}  \put(11,9){\line(0,1){1}}    \put(10,10){\line(0,1){1}}    \put(11,10){\line(0,1){1}}

  \put(13,9){\line(1,0){1}}  \put(12,10){\line(1,0){1}}     \put(14,9){\line(0,1){1}}  \put(13,11){\line(1,0){1}}  \put(12,9){\line(0,1){1}}    \put(13,9){\line(0,1){1}}    \put(12,10){\line(0,1){1}}  \put(13,10){\line(0,1){1}}    \put(14,10){\line(0,1){1}}

  \put(15,9){\line(1,0){1}}  \put(16,9){\line(1,0){1}}  \put(15,10){\line(1,0){1}}     \put(15,11){\line(1,0){1}}  \put(16,11){\line(1,0){1}} \put(16,9){\line(0,1){1}}  \put(17,9){\line(0,1){1}}      \put(16,10){\line(0,1){1}}    \put(17,10){\line(0,1){1}}

  \put(18,9){\line(1,0){1}}  \put(19,9){\line(1,0){1}}  \put(18,10){\line(1,0){1}}    \put(20,9){\line(0,1){1}}  \put(19,11){\line(1,0){1}} \put(19,9){\line(0,1){1}}  \put(18,10){\line(0,1){1}}  \put(19,10){\line(0,1){1}}    \put(20,10){\line(0,1){1}}


  \put(1,6){\line(1,0){1}}  \put(0,7){\line(1,0){1}}    \put(1,7){\line(1,0){1}}  \put(0,8){\line(1,0){1}}  \put(1,8){\line(1,0){1}}  \put(0,6){\line(0,1){1}}    \put(1,6){\line(0,1){1}}  \put(0,7){\line(0,1){1}}  \put(2,7){\line(0,1){1}}

  \put(3,7){\line(1,0){1}}    \put(4,7){\line(1,0){1}}  \put(3,8){\line(1,0){1}}  \put(4,8){\line(1,0){1}}  \put(3,6){\line(0,1){1}}    \put(4,6){\line(0,1){1}}  \put(5,6){\line(0,1){1}}  \put(3,7){\line(0,1){1}}  \put(5,7){\line(0,1){1}}

  \put(6,6){\line(1,0){1}}  \put(7,6){\line(1,0){1}}  \put(6,7){\line(1,0){1}}    \put(7,7){\line(1,0){1}}  \put(6,8){\line(1,0){1}}  \put(7,8){\line(1,0){1}}  \put(7,6){\line(0,1){1}}  \put(6,7){\line(0,1){1}}  \put(8,7){\line(0,1){1}}

  \put(9,6){\line(1,0){1}}  \put(9,7){\line(1,0){1}}    \put(10,7){\line(1,0){1}}  \put(9,8){\line(1,0){1}}  \put(10,8){\line(1,0){1}} \put(11,6){\line(0,1){1}}   \put(10,6){\line(0,1){1}}  \put(9,7){\line(0,1){1}}  \put(11,7){\line(0,1){1}}


  \put(13,6){\line(1,0){1}}  \put(12,7){\line(1,0){1}}    \put(13,7){\line(1,0){1}}  \put(12,8){\line(1,0){1}}  \put(13,8){\line(1,0){1}}  \put(12,6){\line(0,1){1}}    \put(13,6){\line(0,1){1}}  \put(13,7){\line(0,1){1}}

  \put(15,7){\line(1,0){1}}    \put(16,7){\line(1,0){1}}  \put(15,8){\line(1,0){1}}  \put(16,8){\line(1,0){1}}  \put(15,6){\line(0,1){1}}    \put(16,6){\line(0,1){1}}  \put(17,6){\line(0,1){1}}  \put(16,7){\line(0,1){1}}

  \put(19,6){\line(1,0){1}}  \put(18,7){\line(1,0){1}}    \put(19,7){\line(1,0){1}}  \put(19,8){\line(1,0){1}}  \put(18,6){\line(0,1){1}}    \put(19,6){\line(0,1){1}}  \put(18,7){\line(0,1){1}}  \put(19,7){\line(0,1){1}}


  \put(0,4){\line(1,0){1}}    \put(1,4){\line(1,0){1}}  \put(1,5){\line(1,0){1}}  \put(0,3){\line(0,1){1}}    \put(1,3){\line(0,1){1}}  \put(2,3){\line(0,1){1}}  \put(0,4){\line(0,1){1}}  \put(1,4){\line(0,1){1}}

  \put(4,3){\line(1,0){1}}  \put(3,4){\line(1,0){1}}    \put(4,4){\line(1,0){1}}  \put(3,5){\line(1,0){1}}  \put(3,3){\line(0,1){1}}    \put(4,3){\line(0,1){1}}  \put(4,4){\line(0,1){1}}    \put(5,4){\line(0,1){1}}

  \put(6,4){\line(1,0){1}}    \put(7,4){\line(1,0){1}}  \put(6,5){\line(1,0){1}}  \put(6,3){\line(0,1){1}}    \put(7,3){\line(0,1){1}}  \put(8,3){\line(0,1){1}}  \put(7,4){\line(0,1){1}}    \put(8,4){\line(0,1){1}}

  \put(10,3){\line(1,0){1}}  \put(9,4){\line(1,0){1}}    \put(10,4){\line(1,0){1}}  \put(9,3){\line(0,1){1}}    \put(10,3){\line(0,1){1}}  \put(9,4){\line(0,1){1}}  \put(10,4){\line(0,1){1}}    \put(11,4){\line(0,1){1}}

  \put(12,4){\line(1,0){1}}    \put(13,4){\line(1,0){1}}  \put(12,3){\line(0,1){1}}    \put(13,3){\line(0,1){1}}  \put(14,3){\line(0,1){1}}  \put(12,4){\line(0,1){1}}  \put(13,4){\line(0,1){1}}    \put(14,4){\line(0,1){1}}

  \put(15,3){\line(1,0){1}}  \put(16,3){\line(1,0){1}}  \put(15,4){\line(1,0){1}}    \put(16,4){\line(1,0){1}}  \put(15,5){\line(1,0){1}}  \put(16,5){\line(1,0){1}}  \put(16,3){\line(0,1){1}}    \put(16,4){\line(0,1){1}}

  \put(18,3){\line(1,0){1}}  \put(18,4){\line(1,0){1}}    \put(19,4){\line(1,0){1}}  \put(18,5){\line(1,0){1}}  \put(19,5){\line(1,0){1}}  \put(19,3){\line(0,1){1}}  \put(20,3){\line(0,1){1}}  \put(19,4){\line(0,1){1}}


  \put(0,0){\line(1,0){1}}  \put(1,0){\line(1,0){1}}  \put(0,1){\line(1,0){1}}    \put(1,1){\line(1,0){1}}  \put(1,2){\line(1,0){1}}  \put(1,0){\line(0,1){1}}    \put(0,1){\line(0,1){1}}  \put(1,1){\line(0,1){1}}

  \put(3,0){\line(1,0){1}}  \put(3,1){\line(1,0){1}}    \put(4,1){\line(1,0){1}}  \put(4,2){\line(1,0){1}}  \put(4,0){\line(0,1){1}}  \put(5,0){\line(0,1){1}}  \put(3,1){\line(0,1){1}}  \put(4,1){\line(0,1){1}}

  \put(6,0){\line(1,0){1}}  \put(7,0){\line(1,0){1}}  \put(6,1){\line(1,0){1}}    \put(7,1){\line(1,0){1}}  \put(6,2){\line(1,0){1}}  \put(7,0){\line(0,1){1}}    \put(7,1){\line(0,1){1}}    \put(8,1){\line(0,1){1}}

  \put(9,0){\line(1,0){1}}  \put(9,1){\line(1,0){1}}    \put(10,1){\line(1,0){1}}  \put(9,2){\line(1,0){1}}  \put(10,0){\line(0,1){1}}  \put(11,0){\line(0,1){1}}  \put(10,1){\line(0,1){1}}    \put(11,1){\line(0,1){1}}

  \put(12,0){\line(1,0){1}}  \put(13,0){\line(1,0){1}}  \put(12,1){\line(1,0){1}}    \put(13,1){\line(1,0){1}}  \put(13,0){\line(0,1){1}}    \put(12,1){\line(0,1){1}}  \put(13,1){\line(0,1){1}}    \put(14,1){\line(0,1){1}}

  \put(15,0){\line(1,0){1}}  \put(15,1){\line(1,0){1}}    \put(16,1){\line(1,0){1}}  \put(16,0){\line(0,1){1}}  \put(17,0){\line(0,1){1}}  \put(15,1){\line(0,1){1}}  \put(16,1){\line(0,1){1}}    \put(17,1){\line(0,1){1}}

%

\end{picture}

 \caption{The set $\tiles$ of all 34 tilings of a 2 by 2 square with dominoes.  There are 2 with two complete dominoes, 16 with one complete domino and 16 with no complete dominoes. \label{tilespic}}

\end{figure}

\setlength{\unitlength}{24mm}


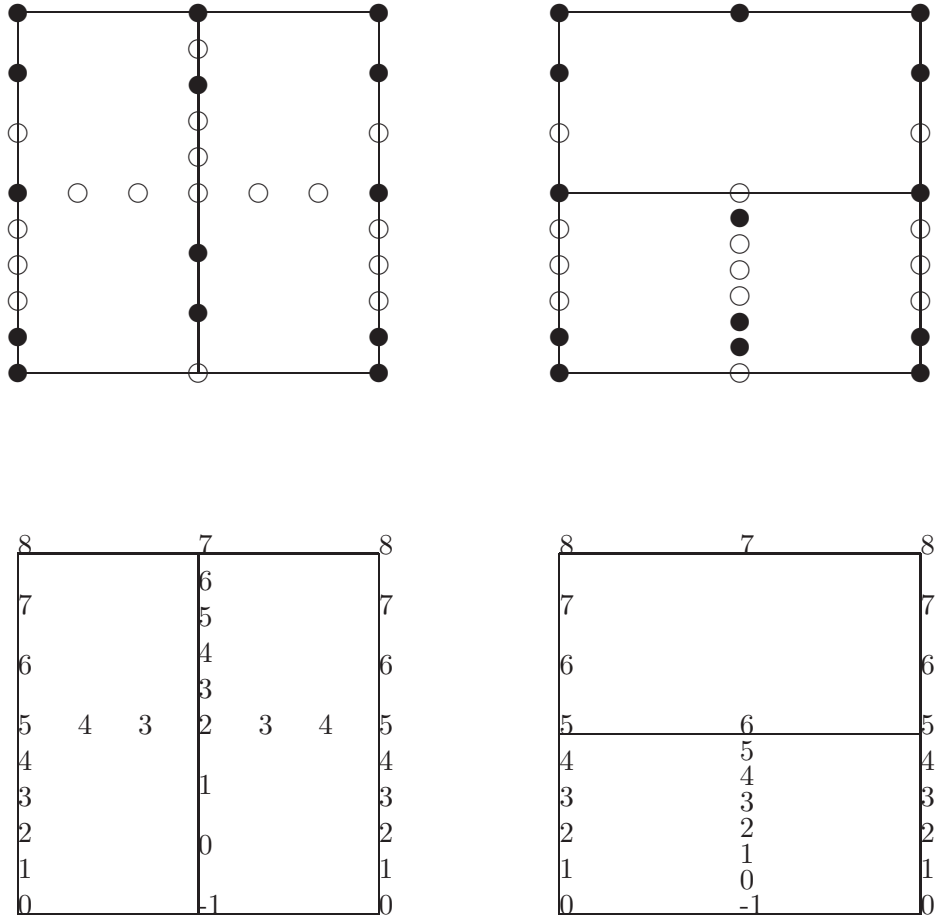
\begin{figure}
\begin{picture}(5,5)

  \put(0,0){\line(1,0){1}}  \put(1,0){\line(1,0){1}}    \put(0,2){\line(1,0){1}}  \put(1,2){\line(1,0){1}}  \put(0,0){\line(0,1){1}}    \put(1,0){\line(0,1){1}}  \put(2,0){\line(0,1){1}}  \put(0,1){\line(0,1){1}}  \put(1,1){\line(0,1){1}}    \put(2,1){\line(0,1){1}}

  \put(3,0){\line(1,0){1}}  \put(3,1){\line(1,0){1}}    \put(3,2){\line(1,0){1}}  \put(4,2){\line(1,0){1}}  \put(3,0){\line(0,1){1}}    \put(4,0){\line(1,0){1}}  \put(5,0){\line(0,1){1}}  \put(3,1){\line(0,1){1}}  \put(4,1){\line(1,0){1}}    \put(5,1){\line(0,1){1}}

\put(1,0){\markminusone}  \put(1,.333){\markzero}  \put(1,.666){\markone}    \put(1,1){\marktwo}
\put(1,1.2){\markthree}  \put(1,1.4){\markfour}  \put(1,1.6){\markfive}  \put(1,1.8){\marksix}  \put(1,2){\markseven}
\put(0,0){\markzero}  \put(0,.2){\markone}  \put(0,.4){\marktwo}  \put(0,.6){\markthree}  \put(0,.8){\markfour}  \put(0,1){\markfive}  \put(0,1.333){\marksix}  \put(0,1.666){\markseven}  \put(0,1){}  \put(0,1){}  \put(0,2){\markeight}

\put(2,0){\markzero}  \put(2,.2){\markone}  \put(2,.4){\marktwo}  \put(2,.6){\markthree}  \put(2,.8){\markfour}  \put(2,1){\markfive}  \put(2,1.333){\marksix}  \put(2,1.666){\markseven}  \put(2,1){}  \put(2,1){}  \put(2,2){\markeight}

\put(0.333,1){\markfour} \put(0.666,1){\markthree} \put(1.333,1){\markthree} \put(1.666,1){\markfour}

\put(4,0){\markminusone}  \put(4,.1428){\markzero}  \put(4,.2857){\markone}    \put(4,.4285){\marktwo}
\put(4,.5714){\markthree}  \put(4,.7142){\markfour}  \put(4,.8571){\markfive}  \put(4,1){\marksix}  \put(4,2){\markseven}
\put(3,0){\markzero}  \put(3,.2){\markone}  \put(3,.4){\marktwo}  \put(3,.6){\markthree}  \put(3,.8){\markfour}  \put(3,1){\markfive}  \put(3,1.333){\marksix}  \put(3,1.666){\markseven}  \put(3,1){}  \put(3,1){}  \put(3,2){\markeight}

\put(5,0){\markzero}  \put(5,.2){\markone}  \put(5,.4){\marktwo}  \put(5,.6){\markthree}  \put(5,.8){\markfour}  \put(5,1){\markfive}  \put(5,1.333){\marksix}  \put(5,1.666){\markseven}  \put(5,1){}  \put(5,1){}  \put(5,2){\markeight}

  \put(0,3){\line(1,0){1}}  \put(1,3){\line(1,0){1}}    \put(0,5){\line(1,0){1}}  \put(1,5){\line(1,0){1}}  \put(0,3){\line(0,1){1}}    \put(1,3){\line(0,1){1}}  \put(2,3){\line(0,1){1}}  \put(0,4){\line(0,1){1}}  \put(1,4){\line(0,1){1}}    \put(2,4){\line(0,1){1}}

  \put(3,3){\line(1,0){1}}  \put(3,4){\line(1,0){1}}    \put(3,5){\line(1,0){1}}  \put(4,5){\line(1,0){1}}  \put(3,3){\line(0,1){1}}    \put(4,3){\line(1,0){1}}  \put(5,3){\line(0,1){1}}  \put(3,4){\line(0,1){1}}  \put(4,4){\line(1,0){1}}    \put(5,4){\line(0,1){1}}

\put(1,3){\newmarkminusone}  \put(1,3.333){\newmarkzero}  \put(1,3.666){\newmarkone}    \put(1,4){\newmarktwo}
\put(1,4.2){\newmarkthree}  \put(1,4.4){\newmarkfour}  \put(1,4.6){\newmarkfive}  \put(1,4.8){\newmarksix}  \put(1,5){\newmarkseven}
\put(0,3){\newmarkzero}  \put(0,3.2){\newmarkone}  \put(0,3.4){\newmarktwo}  \put(0,3.6){\newmarkthree}  \put(0,3.8){\newmarkfour}  \put(0,4){\newmarkfive}  \put(0,4.333){\newmarksix}  \put(0,4.666){\newmarkseven}  \put(0,4){}  \put(0,4){}  \put(0,5){\newmarkeight}

\put(2,3){\newmarkzero}  \put(2,3.2){\newmarkone}  \put(2,3.4){\newmarktwo}  \put(2,3.6){\newmarkthree}  \put(2,3.8){\newmarkfour}  \put(2,4){\newmarkfive}  \put(2,4.333){\newmarksix}  \put(2,4.666){\newmarkseven}  \put(2,4){}  \put(2,4){}  \put(2,5){\newmarkeight}

\put(0.333,4){\newmarkfour} \put(0.666,4){\newmarkthree} \put(1.333,4){\newmarkthree} \put(1.666,4){\newmarkfour}

\put(4,3){\newmarkminusone}  \put(4,3.1428){\newmarkzero}  \put(4,3.2857){\newmarkone}    \put(4,3.42857){\newmarktwo}
\put(4,3.5714){\newmarkthree}  \put(4,3.7142){\newmarkfour}  \put(4,3.8571){\newmarkfive}  \put(4,4){\newmarksix}  \put(4,5){\newmarkseven}
\put(3,3){\newmarkzero}  \put(3,3.2){\newmarkone}  \put(3,3.4){\newmarktwo}  \put(3,3.6){\newmarkthree}  \put(3,3.8){\newmarkfour}  \put(3,4){\newmarkfive}  \put(3,4.333){\newmarksix}  \put(3,4.666){\newmarkseven}  \put(3,4){}  \put(3,4){}  \put(3,5){\newmarkeight}

\put(5,3){\newmarkzero}  \put(5,3.2){\newmarkone}  \put(5,3.4){\newmarktwo}  \put(5,3.6){\newmarkthree}  \put(5,3.8){\newmarkfour}  \put(5,4){\newmarkfive}  \put(5,4.333){\newmarksix}  \put(5,4.666){\newmarkseven}  \put(5,4){}  \put(5,4){}  \put(5,5){\newmarkeight}

\end{picture}
 \caption{The lower pictures show two elements of $A,B \in \tiles$ along with their sets $\halfA$ and $\halfB$. The points in $\halfA$ ($\halfB$) are marked by the value of $g_A$ ($g_B$).  In the upper pictures we chose a coloring of $A$ and $B$.  The adjacency rule for the subshift of finite type allow us to place these colored tiles to be placed next to each other horizontally but not vertically.  \label{ohno}}

\end{figure}

\newpage

\section{The two processes are the same}
\label{same}

In Section \ref{construction} we defined $\dominotilings^*$ and in
Section \ref{subshift} we defined $S$, the state space for our shift of finite type.
Now we will construct a natural bijection between these two spaces.

\begin{definition} \label{defofbijection}
For $(x,c)\in \dominotilings^*$ we define $$(x,c)_{i,j}=(D,\tilde c)\in A$$
by setting $D$ to be
$$T_{{\text left}}^{2i}(T_{{\text down}}^{2j}(x))|_{[0,2] \times [0,2]}$$ and $\tilde c: \colorsD \to \{1,2\}$ by
 $$\tilde c(k)=c(g_x(2i,2j)+k).$$
\end{definition}

\begin{remark} \label{consequencesoftwo} We often consider $(x,c)_{i,j}$ for $(i,j)\in [-N,N]^2.$
This depends on $x$ in $$[-2N,2N+2] \times [-2N,2N+2]$$ and $c$ in
 $$\left\{\min g_x(v),\dots,\max g_x(v)\right\} $$
where the min and max are taken over all $v \in [-2N,2N+2]^2.$
\end{remark}

The definition of $(x,c)_{i,j}$ defines a map from $$M : \dominotilings^* \to A^{\Z^2}$$
by $(M(x,c))_{i,j}=(x,c)_{i,j}.$ Our goal for this section is to prove the following.
\begin{lemma} \label{easter}
$M$ is a shift invariant bijection from $\dominotilings^*$ to $S$.
\end{lemma}

To prove this we will use the following lemmas.

\begin{lemma} \label{peter}
For all $D \in \tiles$
$$\text{Range}(g_D) = Range(g_D)|_{\{0,1,2\} \times [0,2]}.$$
\end{lemma}
\begin{proof}
Since $g_x$ is defined by linearity both $\min g_D$ and $\max g_D$ are achieved on
$\{0, 1, 2\}\times\{0, 1, 2\}.$ By Lemma 3.5 $g_x(i, j) \leq  g_x(i+1, j +2)$ and
$g_x(i, j) \geq g_x(i + 1, j - 2).$ Combined with the first statement these imply that
$$\text{Range}(g_D) = \text{Range}(g_D)|_{\{0,1,2\} \times [0,2]}.$$
\end{proof}

These next two lemmas show how the local adjacency rule implies the global coloring rule. It does so because for any to u, v with $g_x(u) = g_x(v)$ there exists an intermediate sequence $\{w_i\}$ connecting $u$ to $v$ with $g_x(w_i) = g_x(u) = g_x(v)$ that are all colored the same.
\begin{lemma} \label{paul}
If $i \in Z$, $x \in \dominotilings$ and $$g_x(i, j) = g_x(i + 1, j')$$
then $|j-j'|<2$ and there exists $m,n \in \Z$ such that either
\begin{enumerate}
\item $(i,j),(i+1,j') \in [2m,2m+2]\times [2n,2n+2] $or
\item  one of $(i,j)$ and $(i+1,j')$ is in $$[2m,2m+2]\times [2n,2n+2]$$
and the other is in
$$[2m,2m+2] \times [2n-2,2n].$$
\end{enumerate}
Also there exists
$$(i'', j'') \in [2m, 2m + 2] \times [2n] $$ with $g_x(i'', j'') = g_x(i, j).$
\end{lemma}
\begin{proof}
The first statement follows by linearity from Lemma 3.5. Because of this either there exists $n$ such that
\begin{enumerate}
\item  $2n-2<j<2n<j' <2n+2$,
\item $ 2n \leq j, j'  \leq 2n+2$ or
\item $ 2n+2>j>2n>j' >2n-2.$
\end{enumerate}
In the second case then there exists $m \in \Z$ such that
$(i, j), (i + 1, j') \in [2m, 2m + 2] \times [2n, 2n + 2].$ The first and third cases are symmetric so it suffices to only consider the first. As $g_x$ is increasing on $i \times [j, 2n]$ and
$(i + 1) \times [2n, j']$ and $g_x(i, j) = g_x(i + 1, j')$, by the intermediate value theorem we can find the appropriate $(i'', j'') \in [i, i + 1] \times 2n.$
\end{proof}

\begin{lemma} \label{mary}
Given $v,v' \in (\Z\times \R) \cup (\R\times \Z)$ and $x \in \dominotilings$ with
$g_x(v) = g_x(v')$
there exists a sequence $w_0,\dots ,w_k$ such that
\begin{enumerate}
\item $v=w_0$,
\item  $v' =w_k$,
\item $g_x(w_i)=g_x(w_{i+1})$ for all $i=0,\dots,k-1$ and
\item for all $i=0, \dots , k-1$ there exist $m,n \in \Z$ such that
$$w_i,w_{i+1}  \in [2m,2m+2] \times [2n,2n+2].$$
\end{enumerate}
\end{lemma}
\begin{proof}
By Lemma \ref{peter} we can find an appropriate $w_1 = (a, b) \in \Z \times \R.$
By Lemma \ref{ne} we can find $w_i' \in (a+i-1) \times \R$ such that $g_x(w_i') = g_x(v).$
If for some $i$ the points $w_i'$ and $w_{i'+1}$ do not satisfy the last condition,
then Lemma \ref{paul} ensures we can insert an intermediate point $w''$ such
that
\begin{enumerate}
\item $g (w'') = g (v)$, and
\item there exist $m, n \in \Z$ such that
$$w',w'' \in [2m,2m+2] \times [2n,2n+2]$$
and
$$w_i'', w_{i+1}'	\in  [2m, 2m + 2] \times  [2n - 2, 2n]$$
 or vice versa.
 \end{enumerate}
Then Lemma \ref{peter} implies that there exists $i$, $m$ and $n$ such that
$w_i' , v' \in  [2m, 2m + 2]  \times [2n, 2n + 2].$ This completes the proof.
\end{proof}

\begin{pfoflem}{\ref{easter}}
To show that $M$ is a shift invariant bijection we need to establish the following properties.
\begin{itemize}
\item $M(\dominotilings^*)\subset S$,
\item $M$ is 1-1
\item $M$ is shift invariant.
\item $M$ is invertible.
\end{itemize}
The first three follow straight  from the definition of $M$, the coloring rule and the adjacency rule.  To show that $M$ is invertible we first show that any element of $s \in S$ generates a complete domino tiling $x_s$ of the plane.  Then we show that the coloring of $s$ satisfies the coloring rule of Definition \ref{vancouver}.

Any horizontal line segment from $(i,j)$ to $(i+1,j)$ on the boundary of a tile $D$ which bisects a domino has three elements of $\halfD$ in its interior (as $g$ changes by 3 across such an edge), while a horizontal line segment which does not bisect the boundary of a domino has no elements of $\halfA$ in its interior(as $g$ changes by 1 across such an edge).

Likewise any vertical line segment from $(i,j)$ to $(i,j+1)$ on the boundary of a tile $A$ which bisects a domino has zero or six elements of $\halfA$ in its interior (as $g$ changes by 1 or 7 across such an edge) while a vertical line segment which does not bisect the boundary of a domino has two or four elements of $\halfA$ in its interior (as $g$ changes by 3 or 5 across such an edge).
Thus if two tiles can be placed next to each other under the adjacency rule then their domino tilings are consistent and every element of $s$ generates a domino tiling $x_s$ of the plane.

%

Now we verify the coloring rule.  For every $v,v' \in (\Z \times \R) \cup (\R \times \Z)$
with
$g_{x_s}(v)=g_{x_s}(v')$
by Lemma \ref{mary}  there is a sequence $w_i \in \Z^2$ with

\begin{enumerate}
\item $v=w_0$,
\item $v'=w_k$ and
\item $g_{x_s}(w_i)=g_{x_s}(w_{i+1})$ for all $i=0,1,\dots,k-1$ and
\item  for all $i=0,1,\dots,k-1$ there exist $m,n \in \Z$  such that
$$w_i,w_{i+1} \in [2m,2m+2]\times[2n,2n+2].$$
\end{enumerate}

Then the last two conditions along with the adjacency rule imply that
$C(w_i)=C(w_{i+1})$ for every $i=0,\dots,k-1.$ Thus $C(v)=C(v')$ and the coloring rule of Definition \ref{vancouver} is satisfied.
Thus $M^{-1}(s)$ exists and $M$ is a shift invariant bijection.
\end{pfoflem}
%


\section{Entropy} \label{entropy}

 Up until now we have treated $S$ as just a topological object.  Using $M$ we can put a measure on $S$ by looking at the pushforward of $\mu^*$.  To minimize notation we refer to the pushforward of $\mu^*$ as $\mu^*$ as well.  We will show that $\mu^*$ is a measure of maximal entropy for our subshift of finite type.

 Let $\num(N)$ be the number of domino tilings of $[-N,N]^2.$ Then there exists $h>0$ such that
$$\lim_{N \to \infty} \frac1{4N^2} \log \num(N) \to h(\text{domino}).$$

\begin{lemma} \label{topologicalentropy}
The topological entropy of the colored domino shift is four times
the entropy of the domino shift.
\end{lemma}

\begin{proof}
By Remark \ref{consequencesoftwo}
$M(x,c)|_{[-N,N]^2}$ is determined by $$x|_{[-2N,2N+2]^2}$$ and
$$c|_{\{\min g_x(i,j),\dots,\max g_x(i,j)\}}$$
where the min and max are over all $(i,j)\in [-2N,2N+2]^2.$

By Lemma \ref{ne}
$$\{\min g_x(i,j),\dots, \max g_x(i,j)\} \subset [-20N,-20N+20]]$$
for all $(i,j)\in [-2N,2N+2]^2.$
Thus for every $N$
$$\left|S|_{[-N,N]^2}\right| \geq \num(2N+1)$$  and at most
$$\left|S|_{[-N,N]^2}\right| \leq \num(2N+1)2^{40N+21}.$$
The first statement implies that the entropy of the colored domino
tiling is at least four times the entropy of the domino tiling.  The second
implies
\begin{eqnarray*}
\lim_{N \to \infty}\frac{\left|S|_{[-N,N]^2}\right|}{(2N+1)^2}
& \leq &\lim_{N \to \infty}\frac{ \log \num(2N+1)2^{40N+21}  }{(2N+1)^2}\\
 &\leq & \lim_{N \to \infty}\frac{40N+21+\log \num(2N+1)}{4N^2}\\
 &=& \lim_{N \to \infty}\frac{\log \num(2N+1)}{4N^2}\\
 &=& 4h(\text{domino})
\end{eqnarray*}
  which proves the lemma.
\end{proof}

By Lemma \ref{topologicalentropy} the measure $\mu^*$ is a measure of
maximal entropy.  As we could change $\prob^*$ to $\prob \times $
Bernoulli(1/3,2/3)  and not change the entropy of the colored domino
tiling, the subshift of finite type does not have a unique measure of maximal entropy.  But $\mu^*$ is in
some sense the most natural measure of maximal entropy. It is the
measure we get by generating measures $\hat \mu_N$ by putting equal mass on all
colored domino tilings of $[-N,N]^2$ and the taking the
limit as $N$ goes to $\infty$.  We can show that this limit converges by
the uniqueness of the measure of maximal entropy for the domino
tiling.

Let $\mu'$ be any weak limit of the sequence of measures above.  Then $\mu'$ is a measure of maximal entropy on $S$.  Then project $\mu'$ onto $\dominotilings$ to get
a measure $\tilde \mu$ on $\dominotilings$.  The entropy of $\tilde \mu$ is the same as the entropy of $\mu'$.  Thus $\tilde \mu$ is a measure of maximal entropy on $\dominotilings$.  As there is a unique such measure we have that $\tilde \mu=\mu$.

In every of our sequence of measures conditioned on a domino tiling the colors are independent with all colorings equally likely.
We say $B \subset \dominotilings$ is a {\bf cylinder set determined by $[-n,n]^2$} if
$x \in B$ is determined by $x(u)$ for all $ u \in [-n,n]^2$.  We
write $B \in \cyl_{[-n,n]^2}$.
Let $B$ be a cylinder set for dominoes and $c_1$ and $c_2$ be colorings of the colored points in $B$.  Then
$\hat \mu_N(B,c_1)=\hat \mu_N(B,c_2)$.  This carries over to any weak limit.  Thus $\mu'=\mu^*$.

\section{Completely positive entropy}
\label{isk}

In order to show that our subshift has completely positive entropy we take the following steps.

\begin{itemize}
\item We take a domino configuration $x \in \dominotilings$ and a cylinder set $B \in \cyl_{[-n,n]^2}$ and condition on
$$B\cap \tilde x_{2^kn}.$$

\item  Then we show in Lemma \ref{lotsofcycles} that the conformal invariance of the height function implies that for any $n$ and $m$ if $k$ is sufficiently large then for all $B$ and most $x$  the conditional probability  that the union of two domino tilings in $\set$ have at least $m$ cycles that surround $\sq_n$ and are inside $\sq_{2^kn}$ is close to 1.

\item Next in Lemmas \ref{cyclesmeanunpredictable} and \ref{lem:unpredictable} we show that this implies that for all $\epsilon>0$ if $k$ is sufficiently large then for all $j \in \Z$
 \begin{equation}\label{vonn}\prob\left(h_{x'}(-2^kn,-2^kn)=j\ |\ x' \in \set \right)<\epsilon.\end{equation}
\item  Finally we show how (\ref{vonn}) implies that the colored domino process has completely positive entropy.
\end{itemize}

When we defined the height function we arbitrarily chose to set $h_x(0,0)=0$ for all $x$.  Sometimes it is more convenient to chose to set the function to be zero at a point of the form $(-l,-l)$.  For this reason we
define
$$h_{x,l}(i,j)=h_x(i,j)-h_x(-l,-l).$$
Now define
$$ B^*_{x,m,n,k}=\bigg\{(x',y'):\ x' \cup y' \text{ has $\geq m$ cycles surrounding $\sq_n$ and }
 x',y' \in B \cap \tilde x_{2^kn}\bigg\}$$

Here is the key fact about the fluctuations of the height function
that we need in this paper.

\begin{lemma} \label{lotsofcycles}
Given any $n\in \N$, cylinder set $B \in \cyl_{[-n,n]^2}$  and $\delta>0$ there exists $k \in \N$ and
$G \subset \dominotilings$ with $\prob(G)>1-\delta$ such that
for all $x \in G$
$$\prob \times \prob\bigg(B^*_{x,m,n,k}\  | \    x,y \in B \cap \tilde x_{2^kn} \bigg) >1-\delta.$$
\end{lemma}

\begin{proof}
For a point $x' \in B \cap \tilde x_{2^kn}$ we get a sequence of
functions
$$h_{x',n}|_{\sq_{n}},h_{x',2n}|_{\sq_{2n}},\dots,h_{x',2^kn}|_{\sq_{2^kn}}.$$
Thus the set $\tilde x_{2^kn}$ generates a measure on sequences of
functions. To independently sample $x'$ and $y'$ from $B \cap \tilde
x_{2^kn}$ we first independently sample $h_i:\sq_{2^in} \to \Z$ for $i
=0,\dots,k$ and $h'_i:\sq_{2^in} \to \Z$ for $i =0,\dots,k$
 according to this measure.
Then we sample $x'$ and $y'$ such that
$$h_{x',n}=h_0,  \dots , h_{x',2^kn}= h_{k}$$
and
$$h_{y',n}=h'_0, \dots , h_{y',2^kn}=h'_{k}.$$

Fix $\delta>0$ smaller than the $\delta$ in Theorem \ref{cor:kenyon}.
The difference
$$h_x(n,0)=\sum_{i=1}^{n}h_x(i,0)-h_x(i-1,0)
  =\sum_{i=1}^{n}h_{T^{i-1}_{\text{left}}(x)}(1,0)-h_{T^{i-1}_{\text{left}}(x)}(0,0)$$ is an ergodic sum.
Repeated uses of the ergodic theorem show that for most $x$  for all large $l$ the probability $||h_{l} ||_\infty<\delta2^ln$ is high.
This implies that for all $k$ sufficiently large for a set of $x$ of measure at least $1-\delta$ the sequences
$h_0,\dots,h_{k}$ and $h'_0,\dots,h'_{k}$
have the property that with probability $1-\delta/2$
$$\#\left( l| \ ||h_{l} ||_\infty, ||h_{l-1}||_\infty <\delta 2^ln \right)>k(1-\delta). $$

For $l$ with
$$||h_{l} ||_\infty, ||h_{l-1}||_\infty, ||h'_{l} ||_\infty, ||h'_{l-1}||_\infty  <\delta 2^ln$$
and $l$ large enough for Corollary \ref{cor:kenyon} to hold this implies that
there exists $p>0$ such that the probability that $x\cup y$ has a cycle in
$\ann_{l}$ is at least $p$.  This happens independently for all such $l$.

Thus for any $m$ we can choose $K$ such that for $k>K$ the set of $x$ such that
$$\prob \times \prob \bigg(B^*_{x,m,n,k} \  |
   x',y' \in B \cap \tilde x_{2^kn} \bigg) $$ is greater than $1-\delta$.
\end{proof}

\begin{lemma} \label{cyclesmeanunpredictable}
Given $k,m,n\in \N$, $x \in \dominotilings$ and cylinder set $B \in \cyl_{[-n,n]^2}$
\begin{equation}\label{halfpipe}
\mu\times \mu \left(h_{x',2^kn}(0,0)=h_{y',2^kn}(0,0)\bigg|\
 B^*_{x,m,n,k}
\right)<1/\sqrt{m}\end{equation}
\end{lemma}

\begin{proof}
Let $\cycles$ be any set of at least $m$ cycles that surround $\sq_n$ and are inside $\sq_{2^kn}$.  Let  $B_{x,2^kn,\cycles}$ be the set of pairs $(x',y')$ such that $(x',y') \in B \cap \tilde x_{2^kn}$
and that the set of cycles in $x' \cup y'$ that surround $\sq_n$ is $\cycles$.  Then by Theorem \ref{thm:kenyon2} we have that the conditional probability that
\begin{equation}\label{relay}\mu \times \mu \left( h_{x',2^kn}(0,0)=h_{y',2^kn}(0,0)\bigg|
  B_{x,2^kn,\cycles} \right) < 1/\sqrt{m}.
\end{equation}
We have that
 $$B^*_{x,m,n,k}=\bigcup_{\cycles}B_{x,2^kn,\cycles}$$
so by Bayes' rule the conditional probability on the left hand side of (\ref{halfpipe}) is a weighted average of conditional probabilities of the form that are on the left hand side of (\ref{relay}).  As all of these are less than $1/\sqrt{m}$, this proves the lemma.
\end{proof}

\begin{lemma} \label{lem:unpredictable}
Given $n\in \N$, cylinder set $B \in \cyl_{[-n,n]^2}$  and $\epsilon>0$
there exists $K$  and a set $G \subset
\dominotilings$ with $\mu(G)>1-\epsilon$ such that for all $k>K$, all $x \in G$ and all $j \in \Z$

\begin{equation} \label{adams}
\prob\bigg(h_{x',2^kn}(0,0)=j \big|
 x' \in B \cap \tilde x_{2^kn} \bigg)<\epsilon.
 \end{equation}
\end{lemma}

\begin{proof}

Choose $m \in \N$ and $\delta$ such that
$$1/\sqrt{m}+\delta<\epsilon^2.$$

Then we get
\begin{eqnarray}\nonumber
\lefteqn{\prob\bigg(h_{x',2^kn}(0,0)=j \big| x' \in B \cap \tilde x_{2^kn} \bigg)^2}&&\\
&<&\sum_{j' \in\Z} \prob\bigg(h_{x',2^kn}(0,0)=j' \big| x' \in B \cap \tilde x_{2^kn} \bigg)^2 \nonumber\\
 &=&\prob \times \prob \left( h_{x',2^kn}(0,0)=h_{y',2^kn}(0,0) \big | x',y' \in B\cap \tilde x_{2^kn}\right)  \label{newadams}\\
%
 &\leq& \prob \times \prob \left( h_{x',2^kn}(0,0)=h_{y',2^kn}(0,0)\big| B^*_{x,m,n,k}\right)
  \nonumber\\
&&+ \prob \times \prob ((B^*_{x,m,n,k})^C \ | \ x',y' \in B\cap \tilde x_{2^kn}) \label{donovan}\\
&\leq&1/\sqrt{m}+\delta \label{mcnabb} \\
&\leq & \epsilon^2. \nonumber
\end{eqnarray}
where (\ref{donovan}) is an application of Bayes' rule and (\ref{mcnabb}) follows from Lemmas \ref{cyclesmeanunpredictable} and \ref{lotsofcycles}. Taking square roots completes the proof.

\end{proof}

 Now that we have established this lemma we can combine this with Theorem \ref{dominoisbernoulli} to prove that our subshift of finite type has completely positive entropy.

\begin{theorem}
The colored domino tiling $(\dominotilings^*,\prob^*,T_{\text{left}}^*,T_{\text{down}}^*)$ has completely positive entropy.
\end{theorem}

\begin{proof}
Let
$$(\tilde x,c)_{N}=\{(x',c'):\ (x',c')_{(i,j)}=(x,c)_{(i,j)}\ \forall (i,j) \not \in [-N,N]^2 \}.$$
By \cite{C} to prove the theorem it suffices to show  that for all $n$
and $\epsilon>0$ there exists $N$ and a set of  $(x,c)$ of
measure at least $1-\epsilon$ such that for all $E \in
\{1,2\}^{[-20n,20n+2]}$ and cylinder sets $B \in Cyl_{[-n,n]^2}$
$$\bigg|\prob\bigg((x,d): x \in B \text{ and } d \in E \big| (\tilde x,c)_{N} \bigg)-\mu(B)2^{-40n-3}\bigg|<\epsilon.$$

We have that
\begin{eqnarray*}
&&\hspace{-1in}\prob\bigg((x,d): x \in B \text{ and } d \in E \bigg| (\tilde x,c)_{N}  \bigg)\\
&=&\prob\left(x \in B \bigg| \tilde x_{N}  \right)\cdot
\prob\left(d \in E \bigg| (\tilde x,c)_{N} \cap B  \right).
\end{eqnarray*}
By Theorem \ref{dominoisbernoulli} the domino tiling is isomorphic to a Bernoulli shift thus it has completely positive entropy. Thus by \cite{dHS} for any $\delta>0$ we can chose $N$ large enough so that there exists a set of $h$ of measure at least $1-\delta$ such that
\begin{equation} \label{zionisthoodlums}
\bigg|\prob\left(x \in B \big| \tilde x_{N}  \right)-\mu(B)\bigg|<\delta
\end{equation}
for all $B$.

Thus we need to show that for all $\delta>0$ there exists $N$ large enough so that for most $x \in \dominotilings$ and
$c \in \{1,2\}^{\Z}$ and all $E\in \{1,2\}^{[-20n,20n+2]}$ and cylinder sets $B$
\begin{equation}\label{warhol}
\bigg|\prob\big(d \in E \big| (\tilde x,c)_{N} \cap B  \big)-2^{-40n-3}\bigg|<\delta.
\end{equation}

By Lemma \ref{lem:unpredictable} for all $\delta'>0$ we can make $k$ large enough so that for most $x$ and all $B$
$$\max_j\prob\bigg(h_{x,2^kn}(0,0)=j \big| \tilde x_{2^kn} \cap B \bigg)<\delta'.$$
Thus we can choose $\delta'$ small enough so that the weak law of large numbers implies (\ref{warhol}) is satisfied for most $c$. For a sufficiently small choice of $\delta$ combining
(\ref{zionisthoodlums}) and (\ref{warhol}) proves the theorem.
\end{proof}


\section{Not Bernoulli}
\label{notbernoulli}

For $c,d \in \{1,2\}^N$ we define
$$\bar f_N(c,d)=1-\frac{j}{N}$$
where $j$ is the largest number such that there exists subsequences
$$1 \leq n_1 <n_2 < \dots < n_j \leq N$$ and
$$1 \leq m_1 <m_2 < \dots < m_j \leq N$$ such that
$c_{n_i}=d_{m_i}$ for all $i=1,\dots,j$. The same definition holds
in the case that one string has length less than $N$.

A simple combinatorial argument proves the following standard lemma.  The lemma is well known but we include the proof here for the sake of completeness.
\begin{lemma}\label{fbarlemma}
There exists $r>0$ such that for all $N$

\begin{equation} \label{luge}
{\colorprob}\times{\colorprob}(c,d:\ \bar f_N(c,d)<.01)<e^{-rN}.
\end{equation}
\end{lemma}

\begin{proof}
For any set $A \subset \{1,2,\dots, N\}$ there exists a unique increasing subsequence $i^{A}_j$ whose union is the complement of $A$. For any $c,d$ with $\bar f_N(c,d) \leq .01$ there exists $A,B \subset \{1,2,\dots ,N\}$ with $|A|=|B|\leq .01N$ such that $c_{i^A_j}=d_{i^B_j}$ for all $j$.  By Tchebychev's inequality there exists $N_0$ such that for all $N>N_0$ there are at most $2^{N/3}$ subsets of $\{1,2,\dots,N\}$ of cardinality at most $.01N$. For any $c$ and $A,B$ with
$|A|=|B|<.01N$
$${\colorprob}(d:\  c_{i^A_j}=d_{i^B_j} \text{ for all } j)\leq 2^{-.99N}.$$
Then for $r'<.1$ and $N>N_0$
\begin{eqnarray*}
{\colorprob}\times{\colorprob}(c,d:\ \bar f_N(c,d)<.01)
  &\leq& \sum_{A,B}{\colorprob}(d:\  c_{i^A_j}=d_{i^B_j} \text{ for all } j)\\
  &\leq& 2^{2N/3}2^{-.99N}\\
  &<& e^{-r'N}.
\end{eqnarray*}
Since this holds for all $N$ sufficiently large there exists $r>0$ such that
(\ref{luge}) is true for all $N$.
\end{proof}

A slight generalization is that there exists $r>0$ such that for all $N$ sufficiently large
\begin{equation} \label{fbar}
\colorprob \times \colorprob\bigg((c,d):\ \exists a,b \in [-100N,100N]\text{ and } K >N \text{ such that }
   \bar f_K(\sigma^{a}c,\sigma^{b}d)<.01\bigg)<e^{-rN}.
\end{equation}

The main tool that we use in this lemma is that a good $\bar d_{[-N,N]^2}$ matching
generates a good $\bar f$ matching of the colorings.  We make that precise in the following lemma.

\begin{lemma} \label{quadruple}
If
$$\bar d_{[-N,N]^2}\left(\mu|_{(\tilde x,c)},\mu|_{(\tilde y,d)}\right)<.0001$$
then there exists $a,b$ and $K$ such that
 $-100N\leq a,b \leq  100N$  and $K>N$
   $$\bar f_K(\sigma^{a}c,\sigma^{b}d)< .01.$$
\end{lemma}

Before we prove this lemma we first introduce a few definitions and then prove a few quick lemmas.  Note that any $x, y$ and $i$ generate a natural bijection $F$ from $\R$ to $\R$ by
$$F(r)=F_{x,y,i}(r)=g^{-1}_{y}(2i,g_x(2i,r)).$$
\begin{lemma} \label{roy}
$F$ is increasing and
$$(1/7)|t'-t| \leq |F(t')-F(t)| \leq 7 |t'-t|.$$
\end{lemma}

\begin{proof}
By Lemma \ref{ne} both $g_x(i,r)$ and $g_y(i,r)$ are piecewise linear increasing in $r$ and have derivatives between 1 and 7. This implies that $F$ is piecewise linear and increasing with derivative between 1/7 and 7.
\end{proof}

Define
$$\agree_{i,N}=\bigg\{j:\ (x',\sigma^{a}c)_{2i,2j} =
  (y',\sigma^{b}d)_{2i,2j}\bigg\}\cap[-N,N].$$

Let
$$\colors_{x',i,N}= \{g_{x'}(2i,-2N),\dots,g_{x'}(2i,2N+2)\}$$
and
$$\colors_{y',i,N}= \{g_{y'}(2i,-2N),\dots,g_{y'}(2i,2N+2)\}.$$
\begin{lemma}
$|\colors_{x',i,N}|,|\colors_{y',i,N}|\geq 4N+2$.
\end{lemma}

\begin{proof}
This follows from Lemma \ref{ne} as both $g_{x'}(2i,*)$ and $g_{y'}(2i,*)$ have derivative at least 1.
\end{proof}

Let
$$\agreedcolors_{x',i,N}=\left\{   g_{x'}(2i,l):\ l \in 2\agree_{i,N}+[0,2]\right\}\cap \Z$$
and
$$\agreedcolors_{y',i,N}=\left\{   g_{y'}(2i,l):\ l \in 2\agree_{i,N}+[0,2]\right\}\cap \Z.$$

\begin{lemma}
$$\big|\colors_{x',i,N} \setminus \agreedcolors_{x',i,N}\big|<(15)(2N+1-| \agree_{i,N} |).$$
\end{lemma}

\begin{proof}
By Lemma \ref{ne} for each $j \in \{-N,\dots,N \}$
there are at most 15 elements $t$ of
 $\colors_{x',i,N} $ such that there exists $z \in [j,j+2]$ with
 $$g_x(2i,z)=t.$$
Thus there are at most
$15(2N+1-|\agree_{i,N}|) $
elements of
 $$\colors_{x',i,N} \setminus \agreedcolors_{x',i,N}.$$
\end{proof}

This next lemma is the reason for our use of the notation $\agreedcolors_{x',i,N}.$

\begin{lemma} \label{ninetysix}
If $t \in  \agreedcolors_{x',i,N}$ then
 $$F(t) \in  \agreedcolors_{y',i,N}$$
and $c_{t}=d_{F(t)} .$
\end{lemma}

\begin{proof}
By the definitions of $g_{x'}$, $g_{y'}$ and $\agree_{i,N}$,  we have that
\be \label{meb} g_{x'}(2i,*)-g_{y'}(2i,*)|_{[j,j+2]}=0\ee
and
\be g_{x'}^{-1}(2i,2j),g_{y'}^{-1}(2i,2j) \in \Z \label{marygates} \ee
for any $j \in \agree_{i,N}.$  By (\ref{meb}) $F$ is linear with slope 1 on any interval $I$ such that
$$g_{x'}(2i,I) \subset \agree_{i,N}+[0,2].$$  Combined with (\ref{marygates}) this implies that  $F(t)\in  \agreedcolors_{y',i,N}$ if  $t \in  \agreedcolors_{x',i,N}.$  The definition of $\agree_{i,N}$ also implies that
$c_{t}=d_{F(t)} .$
\end{proof}

\begin{pfoflem}{\ref{quadruple}}
If
\begin{equation} \label{ifb}
\bar d_{[-N,N]^2}(\mu|_{(\tilde x,c)},\mu|_{(\tilde y,d)})<.0001
\end{equation}
 then there exists $x'\in \tilde
x_N$ and $y'\in \tilde y_N$ and $a,b \in [-20N,20N]$ such that
\begin{equation} \label{dbargood}
\bar d_{[-N,N]^2}((x',\sigma^{a}c),(y',\sigma^{b}d))<.0001.
\end{equation}
Then we can find $i \in [-N,N]$ such that
\begin{equation} \label{dbargood2}
\bar d_{i\times[-N,N]}((x',\sigma^{a}c),(y',\sigma^{b}d))<.0001.
\end{equation}

Since (\ref{dbargood2}) is satisfied then
 $$|\agree_{i,N}|>(1-.0001)(2N+1).$$

 Then
choose the sequence $n_1,\dots,n_k$ to be elements of
$\agreedcolors_{x',i,N}$ in increasing order and chose the sequence
$m_1,\dots,m_k$ to be elements of $\agreedcolors_{y',i,N}$ in
increasing order.  Then by Lemma \ref{ninetysix} 
$F(n_j)=m_j$ and $c_{a+n_j}=d_{b+m_j}$ for all $j$.

$$k>(1-.0015)|\colors_{x',i,N}|$$ and
$$k>(1-.0015)|\colors_{y',i,N}|.$$ Then set
$$K=\max
(|\colors_{x',i,N}|,|\colors_{y',i,N}|).$$  Thus $K\geq N$.
 This subsequence shows that
$$\bar f_{K}(\sigma^{g_{x'}(2i,-2N)}c,\sigma^{g_{y'}(2i,-2N)}d)<.01.$$
\end{pfoflem}

\begin{theorem}
The colored domino tiling $(\dominotilings^*,\prob^*,T_{\text{left}}^*,T_{\text{down}}^*)$ is not isomorphic to a Bernoulli shift.
\end{theorem}

\begin{proof}
By \cite{dHS} the colored domino tiling is isomorphic to a Bernoulli shift if and only if the integral of
$$d_{[-N,N]^2}(\mu^*|_{(\tilde x,c)},\mu^*|_{(\tilde y,d)})$$
over all $x,y,c$ and $d$ approaches 0 as $N$ goes to $\infty$.  But by the comment after Lemma \ref{fbarlemma} for most quadruples
$x,y,c$ and $d$
$$\bar f_K (\sigma^a c,\sigma^b d)>.01$$
for all $a,b$ and $K$ such that $-100N \leq a,b\leq 100N$ and $K>N$.  By Lemma \ref{quadruple} for those quadruples the integrand is bigger than .0001.  Thus the colored domino tiling is not isomorphic to a Bernoulli shift.
\end{proof}

 \begin{theorem}
 The colored domino tiling $(\dominotilings^*,T_{\text{left}}^*,T_{\text{down}}^*)$ is a subshift of finite type. The measure
 $\prob^*$ on $\dominotilings^*$ is a measure of maximal entropy for which the subshift
 has completely positive entropy
 but is not isomorphic to a Bernoulli shift.
 \end{theorem}

\begin{proof}
This is a combination of Lemma
\ref{topologicalentropy} and Theorems \ref{isk} and
\ref{notbernoulli}\end{proof}

\noindent
{\bf Acknowledgements:}  I would like to thank Rick Kenyon, Bob Burton and Dan Rudolph for helpful conversations on this problem.  I would also like to thank an anonymous referee who had many helpful suggestions on how to improve the presentation.  This work was partially supported by NSF grant DMS-\#086024 and an AMS Centennial Fellowship.

\end{document}